\documentclass[12pt,twoside]{article}
\usepackage[margin=0.8in]{geometry}
\usepackage{amsmath, amssymb, amsfonts, mathtools}
\usepackage{amsthm}
\usepackage{mathrsfs}
\usepackage{stmaryrd}
\usepackage{esint}
\usepackage{eqnalign}
\usepackage{graphicx}
\usepackage{enumitem}
\usepackage[most]{tcolorbox}
\usepackage{float}
\usepackage{listings}
\usepackage{array}
\usepackage{booktabs}
\usepackage{xcolor}
\usepackage{emptypage}
\usepackage{tensor}
\usepackage{array}
\usepackage{tabularx}
\usepackage{fancyhdr}
\usepackage{hyperref}
\usepackage[nameinlink,noabbrev]{cleveref}
\usepackage{titlesec}
\usepackage[T1]{fontenc}
\usepackage[utf8]{inputenc}
\usepackage{lmodern}
\usepackage{titling}   
\usepackage[backend=biber,style=numeric]{biblatex}

\newcommand{\EulerB}[2]{\left\langle {#1 \atop #2} \right\rangle^{\!B}}

\newcommand{\HeadTitle}{The Mellin transforms of $1 / \operatorname{arctanh} x$ and $1 / \sqrt{1-x^2}\,\operatorname{arctanh} x$}
\newcommand{\HeadAuthor}{Luc Ramsès TALLA WAFFO}

\titleformat{\section} {\normalfont\normalsize\bfseries\itshape\centering} {\thesection} {0em} {}

\titleformat{\section}[block]
  {\normalfont\large\bfseries\itshape\centering}
  {§\thesection.}
  {1em}
  {}

\titleformat{\subsection}
  {\normalfont\normalsize\bfseries}
  {\thesubsection}
  {1em}
  {}
	\newtheorem{theorem}{Theorem}[section]
\newtheorem{lemma}[theorem]{Lemma}
\newtheorem{proposition}[theorem]{Proposition}

\crefname{lemma}{lemma}{lemmas}
\Crefname{lemma}{Lemma}{Lemmas}

\crefname{proposition}{proposition}{propositions}
\Crefname{proposition}{Proposition}{Propositions}

\crefname{theorem}{theorem}{theorems}
\Crefname{theorem}{Theorem}{Theorems}

\begin{document}

\thispagestyle{fancy} 

\vspace{0.2cm}

\begin{center}
\Large{The Mellin transforms of $\dfrac{1}{\operatorname{arctanh} x}$ and $\dfrac{1}{\sqrt{1-x^2}\,\operatorname{arctanh} x}$} \end{center} 

\hspace{3cm}
\begin{center}
Luc Ramsès TALLA WAFFO \\
Technische Universität Darmstadt\\
Karolinenplatz 5, 64289 Darmstadt, Germany\\
ramses.talla@stud.tu-darmstadt.de\\
\vspace{0.5cm}
\today
\end{center}

\begin{abstract}
We investigate the Mellin transforms of
\(\dfrac{1}{\operatorname{arctanh} x}\) and
\(\dfrac{1}{\sqrt{1-x^2}\,\operatorname{arctanh} x}\) viewed as compactly supported functions on \(\left(0,1\right)\). These transforms are closely connected
with conjectures on the arithmetic nature of the ratios
\(\dfrac{\zeta(2n+1)}{\pi^{2n+1}}\) and \(\dfrac{\beta(2n)}{\pi^{2n}}\) . While their values at
odd integers were previously studied, the evaluation at even integers
leads to classes of improper integrals that cannot be handled by parity
arguments. Using contour integration techniques, we derive explicit
closed-form expressions involving derivatives of the Riemann zeta and
Dirichlet beta functions, thereby extending earlier results and
providing new analytic tools for the study of related hyperbolic
integrals.
\end{abstract}

\vspace{0.2cm}

{\small\underline{\textbf{Notations:}} Throughout this manuscript, $c_{k,n}$ denotes the coefficient of $x^k$ in the Taylor expansion of $\dfrac{x^{2n+1}}{\sinh^{2n+1}x}$ around 0; $d_{k,n}$ those of $x^k$ in the Taylor expansion of $\dfrac{x^{2n}}{\sinh^{2n}x}$. $\displaystyle\genfrac{\langle}{\rangle}{0pt}{}{m}{k}$ stands for Eulerian of type A; $B_{2n}$ denotes Bernoulli numbers, $E_{2n}$ denotes Euler numbers, $\displaystyle\EulerB{n}{k}$ are Eulerian numbers of type B.}

\vspace{0.2cm}

{\small
\tableofcontents }

\vspace{0.5cm}

\section*{Introduction}
\addcontentsline{toc}{section}{Introduction}

\vspace{0.3cm}

The present work is closely related to recent investigations initiated in
\cite{talla_waffo_integral_2025}, devoted to the study of certain
\emph{perspectives on the arithmetic nature of the ratios}$
\dfrac{\zeta(2n+1)}{\pi^{2n+1}}
\quad \text{and} \quad
\dfrac{\beta(2n)}{\pi^{2n}}$.
In particular, it has been established that there exist polynomials
\(\Xi_n(x), \Lambda_n(x) \in \mathbb{Q}[x]\) such that, for every nonzero
natural integer \(n\), the following integral representations hold
\cite{TallaWaffo2025arxiv2511.02843}:
\[
\frac{\beta(2n)}{\pi^{2n-1}}
=
\int_{0}^{1}
\frac{x\,\Xi_n(x)}{\sqrt{1-x^2}\,\operatorname{arctanh} x}\,dx,
\qquad
\frac{\zeta(2n+1)}{\pi^{2n}}
=
\int_{0}^{1}
\frac{x\,\Lambda_n(x)}{\operatorname{arctanh} x}\,dx.
\]
These representations naturally motivate the following conjectural
statements.
\begin{itemize}
\item
The irrationality of the numbers
\(\dfrac{\zeta(2n+1)}{\pi^{2n+1}}\) for all \(n\) would follow from the
conjecture
\[
\forall\, P \in \mathbb{Z}[x]\setminus\{0\}, \qquad
\int_{0}^{1}
\frac{x\,P(x)}{\operatorname{arctanh} x}\,dx \notin \pi\mathbb{Q}.
\]
\item
Similarly, the irrationality of the numbers
\(\dfrac{\beta(2n)}{\pi^{2n}}\) for all \(n\) would follow from
\[
\forall\, P \in \mathbb{Z}[x]\setminus\{0\}, \qquad
\int_{0}^{1}
\frac{x\,P(x)}{\sqrt{1-x^2}\,\operatorname{arctanh} x}\,dx
\notin \pi\mathbb{Q}.
\]
\end{itemize}

These conjectures were originally approached by studying the structure of
the integrals
\[
\int_{0}^{1} \frac{x^{2n-1}}{\operatorname{arctanh} x}\,dx
\quad \text{and} \quad
\int_{0}^{1}
\frac{x^{2n-1}}{\sqrt{1-x^2}\,\operatorname{arctanh} x}\,dx.
\]
However, a point that has received comparatively little attention is that
the polynomials \(\Xi_n\) and \(\Lambda_n\) are \emph{even polynomials},
that is, they are of the form
\[
a_n x^{2n} + a_{n-1} x^{2n-2} + \cdots + a_1 x^2 + a_0.
\]
By contrast, the conjectures above are formulated for arbitrary integer
polynomials and do not explicitly take this parity constraint into account.

This observation naturally raises the question of how odd-degree
contributions enter the picture. More precisely, one is led to consider the
integrals
\[
\int_{0}^{1} \frac{x^{2n}}{\operatorname{arctanh} x}\,dx
\quad \text{and} \quad
\int_{0}^{1}
\frac{x^{2n}}{\sqrt{1-x^2}\,\operatorname{arctanh} x}\,dx,
\]
which, under suitable conditions, are convergent and also play a role in
the analysis of the above conjectures.

This consideration motivates the introduction of the more general families
of integrals
\[
\int_{0}^{1} \frac{x^{s-1}}{\operatorname{arctanh} x}\,dx
\quad \text{and} \quad
\int_{0}^{1}
\frac{x^{s-1}}{\sqrt{1-x^2}\,\operatorname{arctanh} x}\,dx,
\]
where \(s \in \mathbb{C}\) satisfies \(\Re(s) > 1\).
No convergence issues arise in this range, since the integrands admit
continuous extensions at the endpoints. We denote the resulting functions
of the complex variable \(s\) by \(\Phi_1(s)\) and \(\Phi_2(s)\),
respectively.

Remarkably, these functions admit a natural interpretation as
\emph{Mellin transforms}. To make this precise, we introduce the indicator
function of the interval \((0,1)\),
\[
\mathbf{1}_{(0,1)}(x)
:= H(x)-H(x-1)
=
\begin{cases}
1, & 0 < x < 1,\\
0, & \text{otherwise},
\end{cases}
\]
where \(H\) denotes the Heaviside function. With this notation, the
functions \(\Phi_1\) and \(\Phi_2\) can be written formally as
\[
\Phi_1(s)
:= \int_{0}^{+\infty}
\frac{x^{s-1}}{\operatorname{arctanh} x}\,
\mathbf{1}_{(0,1)}(x)\,dx,
\qquad
\Phi_2(s)
:= \int_{0}^{+\infty}
\frac{x^{s-1}}{\sqrt{1-x^2}\,\operatorname{arctanh} x}\,
\mathbf{1}_{(0,1)}(x)\,dx,
\]
which identifies them explicitly as Mellin transforms of compactly
supported functions. This interpretation as Mellin transform, together with the very common and classical convergence condition \(\Re(s) > 1\) and their connection to $\dfrac{\zeta(2n+1)}{\pi^{2n+1}}
\quad \text{and} \quad
\dfrac{\beta(2n)}{\pi^{2n}}$ are the three key features which motivated us to this study. The values of these two functions at even integers have already been
investigated in \cite{talla_waffo_integral_2025}, where several explicit
instances are also provided.

\vspace{0.3cm}

Second, it is shown that the values of these functions at odd integers
belong to a class of integrals that was deliberately left aside in the
original work. Consequently, their derivation cannot be carried out
using exactly the same methods. Indeed, the evaluation at even integers
leads to integrals of the form
\[
\int_{0}^{\infty} f(x)\,dx,
\]
where \(f\) is an odd function. Such integrals cannot be treated by
exploiting parity arguments, as is customary when \(f\) is even, since
this would trivially yield zero. A different approach is therefore
required, and a new method is developed in the present work to handle
this situation.

The aim and the contribution of this article are consequently threefold:
\begin{itemize}
\item to study two classes of Mellin transforms that are deeply connected
with the arithmetic nature of the ratios $
\dfrac{\zeta(2n+1)}{\pi^{2n+1}}
\quad \text{and} \quad
\dfrac{\beta(2n)}{\pi^{2n}}$,
see \cref{sec:values_at_integers} and \cref{sec:transforms};
\item to provide new perspectives on the study of certain integrals that
may have been overlooked in the remarkable works of Blagouchine,
Adamchik, and others; see \cref{sec:first_integral} and
\cref{sec:second_integral};
\item to develop techniques based on complex analysis that may prove
useful in establishing nontrivial identities arising in combinatorics. See \cref{sec:vanishing_functions}, \cref{prop:coefficient_of_eta_minus_one}, \cref{lemma:from_sum_to_integral}, \cref{prop:vanishing}, \cref{lemma:euler_and_bernoulli}, \cref{prop:coefficient_of_zeta_prime2} and \cref{prop:another_non_trivial_identity}.
\end{itemize}

\vspace{0.5cm}

\section{\hspace{0.3cm}The values at odd integers and some related results}\label{sec:values_at_integers}

\vspace{0.3cm}

\begin{lemma}\label[lemma]{lemma:polynomials}
Consider the polynomials\[P_n(x) := \prod_{i=1}^{2n+1}(x-n-1+i) \qquad \text{and} \qquad Q_n(x):=\prod_{i=1}^{2n}(x-n+i)\]

\begin{enumerate}
	\item The roots of $P_n$ are $0, \pm 1, \pm 2, \pm 3, ..., \pm n$ and those of $Q_n$ are $0, -n, \pm 1, \pm 2, \pm 3, ..., \pm (n-1)$
	\item $P_n(x) = \displaystyle\sum_{k = 0}^n g_{k,n} x^{2k+1}$ and $Q_n(x) = \displaystyle\dfrac{1}{4^n}\sum_{k = 0}^n h_{k,n} (2x+1)^{2k}$ where $g_{k,n}$ and $h_{k,n}$ are some integers
\end{enumerate}
\end{lemma}

\begin{proof}

The products can be expanded in the following manner
\[\begin{cases}
P_n(x) = \displaystyle\prod_{i=1}^{2n+1}(x-n-1+i) = \displaystyle\prod_{i=1}^{n}(x-n-1+i) \times x \times\displaystyle\prod_{i=n+2}^{2n+1}(x-n-1+i)\\
\vspace{0.1cm}\\
Q_n(x)=\displaystyle\prod_{i=1}^{2n}(x-n+i) = \displaystyle\prod_{i=1}^{n-1}(x-n+i) \times x \times \displaystyle\prod_{i=n+1}^{2n-1}(x-n+i) \times (x+n)
\end{cases}\]

After reindexing, we arrive at

\[\begin{cases}
P_n(x) = x\displaystyle\prod_{i=1}^{n}(x-i) \,\,\displaystyle\prod_{i=1}^{n}(x+i)\\
\vspace{0.1cm}\\
Q_n(x)= x(x+n)\displaystyle\prod_{i=1}^{n-1}(x-i) \,\, \displaystyle\prod_{i=1}^{n-1}(x+i)
\end{cases}\]

\vspace{0.5cm}

These expressions are sufficient to determine the sets of roots of each polynomial. In particular, the linear factor \(x\) may be identified with the index \(i = 0\) in the product \(\prod_{i=1}^{n-1}(x - i)\), while the factor \((x + n)\) corresponds to the index \(i = n\) in the product \(\prod_{i=1}^{n-1}(x + i)\). Consequently, these factors can be interpreted as extending the respective index sets of the products to include the boundary values \(i = 0\) and \(i = n\). Consequently,

\[\begin{cases}
P_n(x) = x\displaystyle\prod_{i=1}^{n}(x-i) \,\,\displaystyle\prod_{i=1}^{n}(x+i)\\
\vspace{0.1cm}\\
Q_n(x)= \displaystyle\prod_{i=0}^{n-1}(x-i) \,\, \displaystyle\prod_{i=1}^{n}(x+i) = \displaystyle\prod_{i=1}^{n}(x-i+1) \,\, \displaystyle\prod_{i=1}^{n}(x+i)
\end{cases}\]

\vspace{0.5cm}

The factor of the polynomial $Q_n$ can be multiplied by two and it results in a factor $\dfrac{1}{4^n}$ as follows

\[\begin{cases}
P_n(x) = x\displaystyle\prod_{i=1}^{n}(x-i) \,\,\displaystyle\prod_{i=1}^{n}(x+i)\\
\vspace{0.1cm}\\
Q_n(x)= \dfrac{1}{4^n}\displaystyle\prod_{i=1}^{n}(2x-2i+2) \,\, \displaystyle\prod_{i=1}^{n}(2x+2i)
\end{cases}\]

\vspace{0.5cm}

A well-known remarkable identity yields

\[\begin{cases}
P_n(x) = x\displaystyle\prod_{i=1}^{n}(x^2-i^2) \\
\vspace{0.1cm}\\
Q_n(x)= \dfrac{1}{4^n}\displaystyle\prod_{i=1}^{n}((2x+1)^2-(2i-1)^2)
\end{cases}\]

\vspace{0.5cm}

The polynomial \(\displaystyle \prod_{i=1}^{n}(x^2 - i^2)\) is a product of factors depending on \(x^2\), while \(Q_n\) is a product of factors depending on \((2x+1)^2\). When expanded, polynomials of this type contain no odd powers of \(x\) or of \((2x+1)\). Consequently, they admit representations of the following forms:

\[\begin{cases}
P_n(x) = x\displaystyle\sum_{k=0}^{n}g_{k,n}x^{2k} \\
\vspace{0.1cm}\\
Q_n(x)= \dfrac{1}{4^n}\displaystyle\sum_{k=0}^{n}h_{k,n}(2x+1)^{2k}
\end{cases}\]

Recurrence relations of the coefficients can be found by letting 
\[R_n (x) := \displaystyle\prod_{i=1}^{n}(x^2-i^2) \qquad \text{and} \qquad S_n(x) := \displaystyle\prod_{i=1}^{n}((2x+1)^2-(2i+1)^2) \]

$R_n$ and $S_n$ satisfy

\[\begin{cases}
R_n(x) = (x^2-n^2) R_{n-1}(x) \\
\vspace{0.1cm}\\
S_n(x)=(x^2-(2n-1)^2) S_{n-1}(x)
\end{cases}\]

This gives

\[\begin{cases}
\displaystyle\sum_{k=0}^{n}g_{k,n}x^{2k} = (x^2-n^2) \displaystyle\sum_{k=0}^{n-1}g_{k,n-1}x^{2k} \\
\vspace{0.1cm}\\
\displaystyle\sum_{k=0}^{n}h_{k,n}(2x+1)^{2k}=\displaystyle\sum_{k=0}^{n-1}h_{k,n-1}(2x+1)^{2k}-(2n-1)^2) \displaystyle\sum_{k=0}^{n-1}h_{k,n-1}(2x+1)^{2k}
\end{cases}\]

\[\Downarrow\]

\[\begin{cases}
\displaystyle\sum_{k=0}^{n}g_{k,n}x^{2k} = \displaystyle\sum_{k=0}^{n-1}g_{k,n-1}\,x^{2k+2} - n^2\displaystyle\sum_{k=0}^{n-1}g_{k,n-1}\,x^{2k} \\
\vspace{0.1cm}\\
\displaystyle\sum_{k=0}^{n}h_{k,n}(2x+1)^{2k}=\displaystyle\sum_{k=0}^{n-1}h_{k,n-1}\,(2x+1)^{2k+2}-(2n-1)^2 \displaystyle\sum_{k=0}^{n-1}h_{k,n-1}\,(2x+1)^{2k}
\end{cases}\]

\[\Downarrow\]

\[\begin{cases}
\displaystyle\sum_{k=0}^{n}g_{k,n}x^{2k} = \displaystyle\sum_{k=1}^{n}g_{k-1,n-1}\,x^{2k} - n^2\displaystyle\sum_{k=0}^{n-1}g_{k,n-1}\,x^{2k} \\
\vspace{0.1cm}\\
\displaystyle\sum_{k=0}^{n}h_{k,n}(2x+1)^{2k}=\displaystyle\sum_{k=1}^{n}h_{k-1,n-1}\,(2x+1)^{2k}-(2n-1)^2 \displaystyle\sum_{k=0}^{n-1}h_{k,n-1}\,(2x+1)^{2k}
\end{cases}\]

\vspace{0.5cm}

It follows that

\[\boxed{\begin{cases}
g_{k,n} = g_{k-1,n-1} - n^2\,g_{k,n-1} \\
\vspace{0.1cm}\\
h_{k,n}=h_{k-1,n-1}-(2n-1)^2 h_{k,n-1} \\
g_{k,n} = h_{k,n} = 0 \qquad \text{if} \qquad k > n \quad \text{or} \quad k < 0 \\
g_{0,0} = h_{0,0} = 1 \\
g_{k,1} = h_{k,1} = (-1)^{k+1} 
\end{cases}}\]
\end{proof}

\begin{proposition}\label[proposition]{prop:values_at_odd_integers}

\[\forall n \in \mathbb{N}^*, \exists \begin{pmatrix}
M_{0,n} \\
M_{1,n} \\
M_{2,n} \\
M_{3,n} \\
\vdots  \\
M_{n,n} 
\end{pmatrix}, \begin{pmatrix}
N_{0,n} \\
N_{1,n} \\
N_{2,n} \\
N_{3,n} \\
\vdots  \\
N_{n,n} 
\end{pmatrix} \in \mathbb{Q}^{n+1}:\]
\[\begin{cases}
\Phi_1(2n+1) = \displaystyle\sum_{i = 0}^n M_{i,n} \, \eta'(-2i-1)\\
\Phi_2(2n+1) = \displaystyle\sum_{i = 0}^n N_{i,n} \, \beta'(-2i)
\end{cases}\] where $\eta$ and $\beta$ are respectively the Dirichlet's eta and beta functions.
\end{proposition}

\begin{proof}
\[\begin{cases}
\Phi_1(2n+1) = \displaystyle\int_0^1 \dfrac{x^{2n}}{\operatorname{arctanh}x}\,dx\\
\hspace{0.5cm}\\
\Phi_2(2n+1) = \displaystyle\int_0^1 \dfrac{x^{2n}}{\sqrt{1-x^2}\operatorname{arctanh}x}\,dx\\
\end{cases}\]

The change of variable $x = \tanh u$ yields

\[\begin{cases}
\Phi_1(2n+1) = \displaystyle\int_0^{+\infty} \dfrac{\tanh^{2n}u}{u}\dfrac{1}{\cosh^2 u}\,du\\
\hspace{0.5cm} \\
\Phi_2(2n+1) = \displaystyle\int_0^{+\infty} \dfrac{\tanh^{2n}u}{u}\dfrac{1}{\cosh u}\,du\\
\end{cases}\]

Bearing in mind that $\tanh^2 u = 1 - \dfrac{1}{\cosh^2u}$, the Newton's binomial yields \[\tanh^{2n}u = \displaystyle\sum_{k=0}^{n}(-1)^k \binom{n}{k}\dfrac{1}{\cosh^{2k}u}\] Substituting this identity, we obtain

\[\begin{cases}
\Phi_1(2n+1) = \displaystyle\int_0^{+\infty} \dfrac{1}{u}\sum_{k=0}^{n}(-1)^k \binom{n}{k}\dfrac{1}{\cosh^{2k+2}u}\,du\\
\hspace{0.5cm} \\
\Phi_2(2n+1) = \displaystyle\int_0^{+\infty} \dfrac{1}{u}\sum_{k=0}^{n}(-1)^k \binom{n}{k}\dfrac{1}{\cosh^{2k+1}u}\,du\\
\end{cases}\]

This, in turn, implies the following identities after invoking the definition of \(\cosh\) and multiplying both the numerator and the denominator by appropriate powers of \(e\) so that each denominator can be expressed as a power of \((1 + e^{-2k})\).

\[\begin{cases}
\Phi_1(2n+1) = \displaystyle\int_0^{+\infty} \dfrac{1}{u}\sum_{k=0}^{n}(-1)^k \binom{n}{k}\dfrac{2^{2k+2} e^{-2(k+1)u}}{(1+e^{-2u})^{2k+2}}\,du\\
\hspace{0.5cm} \\
\Phi_2(2n+1) = \displaystyle\int_0^{+\infty} \dfrac{1}{u}\sum_{k=0}^{n}(-1)^k \binom{n}{k}\dfrac{2^{2k+1} e^{-(2k+1)u}}{(1+e^{-2u})^{2k+1}}\,du\\
\end{cases}\]

The well-known binomial geometric series $\dfrac{1}{(1-x)^{n+1}} = \displaystyle\sum_{l = 0}^{\infty}\binom{n+l}{n}x^l$ leads to

\[\begin{cases}
\Phi_1(2n+1) = \displaystyle\int_0^{+\infty} \dfrac{1}{u}\sum_{k=0}^{n}(-1)^k \binom{n}{k}2^{2k+2} e^{-2(k+1)u}\displaystyle\sum_{l = 0}^{\infty}(-1)^l\binom{2k+1+l}{2k+1}e^{-2lu}\,du\\
\hspace{0.5cm} \\
\Phi_2(2n+1) = \displaystyle\int_0^{+\infty} \dfrac{1}{u}\sum_{k=0}^{n}(-1)^k \binom{n}{k}2^{2k+1} e^{-(2k+1)u}\displaystyle\sum_{l = 0}^{\infty}(-1)^l\binom{2k+l}{2k}e^{-2lu}\,du\\
\end{cases}\]

Interestingly, the binomial coefficients $\displaystyle\binom{2k+l}{2k}$ and $\displaystyle\binom{2k+1+l}{2k+1}$ can be written in products forms as follows

\[\begin{cases}
\Phi_1(2n+1) = \displaystyle\int_0^{+\infty} \dfrac{1}{u}\sum_{k=0}^{n}(-1)^k \binom{n}{k}2^{2k+2} e^{-2(k+1)u}\displaystyle\sum_{l = 0}^{\infty}(-1)^l\dfrac{1}{(2k+1)!}\prod_{i=1}^{2k+1}(l+i) \,\,e^{-2lu}\,du\\
\hspace{0.5cm} \\
\Phi_2(2n+1) = \displaystyle\int_0^{+\infty} \dfrac{1}{u}\sum_{k=0}^{n}(-1)^k \binom{n}{k}2^{2k+1} e^{-(2k+1)u}\displaystyle\sum_{l = 0}^{\infty}(-1)^l\dfrac{1}{(2k)!}\prod_{i=1}^{2k}(l+i)\,\,e^{-2lu}\,du\\
\end{cases}\]

Or equivalently

\[\begin{cases}
\Phi_1(2n+1) = \displaystyle\int_0^{+\infty} \dfrac{1}{u}\sum_{k=0}^{n}(-1)^k \binom{n}{k}\dfrac{2^{2k+2}}{(2k+1)!}\displaystyle\sum_{l = 0}^{\infty}(-1)^l\prod_{i=1}^{2k+1}(l+i) \,\,e^{-2(k+l+1)u}\,du\\
\hspace{0.5cm} \\
\Phi_2(2n+1) = \displaystyle\int_0^{+\infty} \dfrac{1}{u}\sum_{k=0}^{n}(-1)^k \binom{n}{k}\dfrac{2^{2k+1}}{(2k)!}\displaystyle\sum_{l = 0}^{\infty}(-1)^l\prod_{i=1}^{2k}(l+i)\,\,e^{-(2l+2k+1)u}\,du\\
\end{cases}\]

After reindexing $l$, we get

\[\begin{cases}
\Phi_1(2n+1) = \displaystyle\int_0^{+\infty} \dfrac{1}{u}\sum_{k=0}^{n}(-1)^k \binom{n}{k}\dfrac{2^{2k+2}}{(2k+1)!}\displaystyle\sum_{l = k+1}^{\infty}(-1)^{l-k-1}\prod_{i=1}^{2k+1}(l-k-1+i) \,\,e^{-2lu}\,du\\
\hspace{0.5cm} \\
\Phi_2(2n+1) = \displaystyle\int_0^{+\infty} \dfrac{1}{u}\sum_{k=0}^{n}(-1)^k \binom{n}{k}\dfrac{2^{2k+1}}{(2k)!}\displaystyle\sum_{l = k}^{\infty}(-1)^{l-k}\prod_{i=1}^{2k}(l-k+i)\,\,e^{-(2l+1)u}\,du\\
\end{cases}\]

The first property of \cref{lemma:polynomials} guarantees that the index \(l\) may effectively start at \(1\) in the first sum and at \(0\) in the second, since the corresponding coefficients vanish identically; more precisely, this holds for all \(l \in \llbracket 1, k \rrbracket\) and for all \(l \in \llbracket 1, k-1 \rrbracket\), respectively. Consequently,

\[\begin{cases}
\Phi_1(2n+1) = \displaystyle\int_0^{+\infty} \dfrac{1}{u}\sum_{k=0}^{n}(-1)^k \binom{n}{k}\dfrac{2^{2k+2}}{(2k+1)!}\displaystyle\sum_{l = 1}^{\infty}(-1)^{l-k-1}\prod_{i=1}^{2k+1}(l-k-1+i) \,\,e^{-2lu}\,du\\
\hspace{0.5cm} \\
\Phi_2(2n+1) = \displaystyle\int_0^{+\infty} \dfrac{1}{u}\sum_{k=0}^{n}(-1)^k \binom{n}{k}\dfrac{2^{2k+1}}{(2k)!}\displaystyle\sum_{l = 0}^{\infty}(-1)^{l-k}\prod_{i=1}^{2k}(l-k+i)\,\,e^{-(2l+1)u}\,du\\
\end{cases}\]

The second property of \cref{lemma:polynomials} gives

\[\begin{cases}
\Phi_1(2n+1) = \displaystyle\int_0^{+\infty} \dfrac{1}{u}\sum_{k=0}^{n}(-1)^k \binom{n}{k}\dfrac{2^{2k+2}}{(2k+1)!}\displaystyle\sum_{l = 1}^{\infty}(-1)^{l-k-1}\sum_{i = 0}^k g_{i,k} l^{2i+1} \,\,e^{-2lu}\,du\\
\hspace{0.5cm} \\
\Phi_2(2n+1) = \displaystyle\int_0^{+\infty} \dfrac{1}{u}\sum_{k=0}^{n}(-1)^k \binom{n}{k}\dfrac{2^{2k+1}}{(2k)!}\displaystyle\sum_{l = 0}^{\infty}(-1)^{l-k}\dfrac{1}{4^k}\sum_{i = 0}^k h_{i,k} (2l+1)^{2i}e^{-(2l+1)u}\,du\\
\end{cases}\]

After simplification

\[\begin{cases}
\Phi_1(2n+1) = \displaystyle\int_0^{+\infty} \dfrac{1}{u}\sum_{k=0}^{n} \binom{n}{k}\dfrac{2^{2k+2}}{(2k+1)!}\displaystyle\sum_{l = 1}^{\infty}(-1)^{l-1}\sum_{i = 0}^k g_{i,k}\, l^{2i+1} \,\,e^{-2lu}\,du\\
\hspace{0.5cm} \\
\Phi_2(2n+1) = \displaystyle\int_0^{+\infty} \dfrac{1}{u}\sum_{k=0}^{n} \binom{n}{k}\dfrac{2}{(2k)!}\displaystyle\sum_{l = 0}^{\infty}(-1)^{l}\sum_{i = 0}^k h_{i,k}\, (2l+1)^{2i}e^{-(2l+1)u}\,du\\
\end{cases}\]

Remarkably, the sums after the term $ 1/u$ are in fact respectively expansions of $\tanh^{2n+1}u \dfrac{1}{\cosh^2 u}$ and $\tanh^{2n}u \dfrac{1}{\cosh u}$. Formally, one has

\[\begin{cases}
\tanh^{2n+1}u \dfrac{1}{\cosh^2 u} = \displaystyle\sum_{k=0}^{n} \binom{n}{k}\dfrac{2^{2k+2}}{(2k+1)!}\displaystyle\sum_{l = 1}^{\infty}(-1)^{l-1}\sum_{i = 0}^k g_{i,k}\, l^{2i+1} \,\,e^{-2lu}\\
\hspace{0.5cm} \\
\tanh^{2n}u \dfrac{1}{\cosh u} = \displaystyle\sum_{k=0}^{n} \binom{n}{k}\dfrac{2}{(2k)!}\displaystyle\sum_{l = 0}^{\infty}(-1)^{l}\sum_{i = 0}^k h_{i,k}\, (2l+1)^{2i}e^{-(2l+1)u}\\
\end{cases}\]

Evaluating at $u = 0$ yields the following regularizations

\[\begin{cases}
\displaystyle\sum_{k=0}^{n} \binom{n}{k}\dfrac{2^{2k+2}}{(2k+1)!}\displaystyle\sum_{l = 1}^{\infty}(-1)^{l-1}\sum_{i = 0}^k g_{i,k}\, l^{2i+1} = 0\\
\hspace{0.5cm} \\
\displaystyle\sum_{k=0}^{n} \binom{n}{k}\dfrac{2}{(2k)!}\displaystyle\sum_{l = 0}^{\infty}(-1)^{l}\sum_{i = 0}^k h_{i,k}\, (2l+1)^{2i}=0\\
\end{cases}\]

After multiplying both sides of the first equality by $e^{-2u}$ and those of the second by $e^{-u}$, we get

\[\begin{cases}
\displaystyle\sum_{k=0}^{n} \binom{n}{k}\dfrac{2^{2k+2}}{(2k+1)!}\displaystyle\sum_{l = 1}^{\infty}(-1)^{l-1}\sum_{i = 0}^k g_{i,k}\, l^{2i+1} e^{-2u} = 0\\
\hspace{0.5cm} \\
\displaystyle\sum_{k=0}^{n} \binom{n}{k}\dfrac{2}{(2k)!}\displaystyle\sum_{l = 0}^{\infty}(-1)^{l}\sum_{i = 0}^k h_{i,k}\, (2l+1)^{2i}e^{-2u}=0\\
\end{cases}\]

The integrals become then

\[\begin{cases}
\Phi_1(2n+1) = \displaystyle\int_0^{+\infty} \dfrac{1}{u}\sum_{k=0}^{n} \binom{n}{k}\dfrac{2^{2k+2}}{(2k+1)!}\displaystyle\sum_{l = 1}^{\infty}(-1)^{l-1}\sum_{i = 0}^k g_{i,k}\, l^{2i+1} \,\,\left[e^{-2lu} - e^{-2u}\right]\,du\\
\hspace{0.5cm} \\
\Phi_2(2n+1) = \displaystyle\int_0^{+\infty} \dfrac{1}{u}\sum_{k=0}^{n} \binom{n}{k}\dfrac{2}{(2k)!}\displaystyle\sum_{l = 0}^{\infty}(-1)^{l}\sum_{i = 0}^k h_{i,k}\, (2l+1)^{2i}\left[e^{-(2l+1)u} - e^{-u}\right]\,du\\
\end{cases}\]

Interchanging summation and integration allows us to write

\[\begin{cases}
\Phi_1(2n+1) = \displaystyle\sum_{k=0}^{n} \binom{n}{k}\dfrac{2^{2k+2}}{(2k+1)!}\displaystyle\sum_{l = 1}^{\infty}(-1)^{l-1}\sum_{i = 0}^k g_{i,k}\, l^{2i+1} \,\,\int_0^{+\infty} \dfrac{1}{u}\left[e^{-2lu} - e^{-2u}\right]\,du\\
\hspace{0.5cm} \\
\Phi_2(2n+1) = \displaystyle\sum_{k=0}^{n} \binom{n}{k}\dfrac{2}{(2k)!}\displaystyle\sum_{l = 0}^{\infty}(-1)^{l}\sum_{i = 0}^k h_{i,k}\, (2l+1)^{2i}\int_0^{+\infty} \dfrac{1}{u}\left[e^{-(2l+1)u} - e^{-u}\right]\,du\\
\end{cases}\]

The Frullani's integral $\displaystyle\int_0^{+\infty}\dfrac{e^{-ax}-e^{-bx}}{x}dx = \ln b/a$ yields

\[\begin{cases}
\Phi_1(2n+1) = \displaystyle\sum_{k=0}^{n} \binom{n}{k}\dfrac{2^{2k+2}}{(2k+1)!}\displaystyle\sum_{l = 1}^{\infty}(-1)^{l-1}\sum_{i = 0}^k g_{i,k}\, l^{2i+1} \,\,(-\ln l)\\
\hspace{0.5cm} \\
\Phi_2(2n+1) = \displaystyle\sum_{k=0}^{n} \binom{n}{k}\dfrac{2}{(2k)!}\displaystyle\sum_{l = 0}^{\infty}(-1)^{l}\sum_{i = 0}^k h_{i,k}\, (2l+1)^{2i}(-\ln(2l+1))\\
\end{cases}\]

Now, $i$ runs up to $k$, which in turn runs up to $n$. This means $i$ runs up to $n$. Taking this observation into account and rearranging, we have

\[\begin{cases}
\Phi_1(2n+1) = \displaystyle\sum_{i = 0}^n\sum_{k=i}^{n} \binom{n}{k}\dfrac{2^{2k+2}}{(2k+1)!}\displaystyle g_{i,k}\,\sum_{l = 1}^{\infty}(-1)^{l-1} l^{2i+1} \,\,(-\ln l)\\
\hspace{0.5cm} \\
\Phi_2(2n+1) = \displaystyle\sum_{i = 0}^n\sum_{k=i}^{n} \binom{n}{k}\dfrac{2}{(2k)!}\displaystyle h_{i,k}\, \sum_{l = 0}^{\infty}(-1)^{l}(2l+1)^{2i}(-\ln(2l+1))\\
\end{cases}\]

The latter series are respectively the derivatives of the Dirichlet eta and beta function. We conclude that

\[\begin{cases}
\Phi_1(2n+1) = \displaystyle\sum_{i = 0}^n\left[\sum_{k=i}^{n} \binom{n}{k}\dfrac{2^{2k+2}}{(2k+1)!}\displaystyle g_{i,k}\right]\,\eta'(-2i-1)\\
\hspace{0.5cm} \\
\Phi_2(2n+1) = \displaystyle\sum_{i = 0}^n\left[\sum_{k=i}^{n} \binom{n}{k}\dfrac{2}{(2k)!}\displaystyle h_{i,k}\,\right] \beta'(-2i)\\
\end{cases}\]
This concludes the proof.
\end{proof}

For the first values of $n$, one has

\[
\begin{cases}
\displaystyle
\int_0^1 \dfrac{x^{2}}{\operatorname{arctanh}x}\,dx
= \dfrac{4}{3}\,\eta'(-1) + \dfrac{8}{3}\,\eta'(-3),\\[10pt]
\hspace{0.5cm} \\
\displaystyle
\int_0^1 \dfrac{x^{2}}{\sqrt{1-x^2}\,\operatorname{arctanh}x}\,dx
= \beta'(0) + \beta'(-2).
\end{cases}
\]

\vspace{0.3cm}

\[
\begin{cases}
\displaystyle
\int_0^1 \dfrac{x^{4}}{\operatorname{arctanh}x}\,dx
= \dfrac{4}{5}\,\eta'(-1)
+ \dfrac{8}{3}\,\eta'(-3)
+ \dfrac{8}{15}\,\eta'(-5),
\\
\hspace{0.5cm}\\
\displaystyle
\int_0^1 \dfrac{x^{4}}{\sqrt{1-x^2}\,\operatorname{arctanh}x}\,dx
= \dfrac{3}{4}\,\beta'(0)
+ \dfrac{7}{6}\,\beta'(-2)
+ \dfrac{1}{12}\,\beta'(-4).
\end{cases}
\]

\vspace{0.3cm}

\[
\begin{cases}
\displaystyle
\int_0^1 \dfrac{x^{6}}{\operatorname{arctanh}x}\,dx
= \displaystyle\frac{4}{7}\,\eta'(-1) + \frac{112}{45}\,\eta'(-3) + \frac{8}{9}\,\eta'(-5) + \frac{16}{315}\,\eta'(-7),\\
\hspace{0.1cm}\\
\displaystyle
\int_0^1 \dfrac{x^{6}}{\sqrt{1-x^2}\,\operatorname{arctanh}x}\,dx
= \displaystyle\frac{5}{8}\,\beta'(0) + \frac{439}{360}\,\beta'(-2) + \frac{11}{72}\,\beta'(-4) + \frac{1}{360}\,\beta'(-6).
\end{cases}
\]

\vspace{0.3cm}

\[
\begin{cases}
\displaystyle
\int_0^1 \frac{x^{8}}{\operatorname{arctanh}x}\,dx
= \frac{4}{9}\,\eta'(-1) + \frac{6\,544}{2\,835}\,\eta'(-3) + \frac{152}{135}\,\eta'(-5) + \frac{16}{135}\,\eta'(-7) + \frac{8}{2\,835}\,\eta'(-9),\\
\hspace{0.1cm}\\
\displaystyle
\int_0^1 \frac{x^{8}}{\sqrt{1-x^2}\,\operatorname{arctanh}x}\,dx
= \frac{35}{64}\,\beta'(0) + \frac{1\,247}{1\,008}\,\beta'(-2) + \frac{301}{1\,440}\,\beta'(-4) + \frac{1}{144}\,\beta'(-6) + \frac{1}{20\,160}\,\beta'(-8).
\end{cases}
\]

\vspace{1cm}

From these identities, we deduce that the coefficient of \(\eta'(-1)\) in \(\Phi_1(2n+1)\) is always equal to \(\dfrac{4}{2n+1}\). Indeed, it suffices to show that
\[
\sum_{k=0}^{n} \binom{n}{k}\frac{2^{2k}}{(2k+1)!}\, g_{0,k} = \frac{1}{2n+1}.
\]
By a straightforward induction argument, one readily verifies that
\[
g_{0,k} = (-1)^k (k!)^2.
\]

\vspace{0.5cm}

\begin{proposition}\label[proposition]{prop:coefficient_of_eta_minus_one}
\[\forall n \in \mathbb{N}_0, \,\,\sum_{k=0}^{n} \binom{n}{k}4^k(-1)^k \dfrac{1}{(2k+1)!}\, (k!)^2 = \frac{1}{2n+1}\]
\end{proposition}
\begin{proof}

Let \[S_n := \sum_{k=0}^{n} \binom{n}{k}4^k(-1)^k \dfrac{1}{(2k+1)!}\, (k!)^2\]

The term $\dfrac{1}{(2k+1)!}\, (k!)^2$ is recognized with the Euler's $\Gamma$ and B-functions as follows

\[\dfrac{1}{(2k+1)!}\, (k!)^2 = \frac{\Gamma(k+1)\Gamma(k+1)}{\Gamma(2k+2)} = B(k+1, k+1) = \int_0^1 x^k\,(1-x)^k\,dx\]

Consequently

\[S_n = \sum_{k=0}^{n} \binom{n}{k}4^k(-1)^k \int_0^1 x^k\,(1-x)^k\,dx = \int_0^1\sum_{k=0}^{n} \binom{n}{k}4^k(-1)^k  x^k\,(1-x)^k\,dx\]

Summation and integration can be trivially swapped here since the sum is finite. After simplification

\[S_n = \int_0^1\sum_{k=0}^{n} \binom{n}{k}(4x^2-4x)^k\,dx = \int_0^1\sum_{k=0}^{n} \binom{n}{k}(4x^2-4x)^k (1)^{n-k}\,dx\]

The Newton's binomial formula is easily recognized. Thus

\[S_n = \int_0^1(4x^2-4x+1)^n \, dx = \int_0^1(2x-1)^{2n} \, dx\] 

The last integral is easily evaluated. This concludes the proof.
\end{proof}

\vspace{1cm}

Although the results established in \cref{prop:values_at_odd_integers} are correct, the proof
presented therein does not fully meet the standards of mathematical
rigor, due primarily to the lack of a proper justification for the
interchange of summation and integration. In
\cref{sec:first_integral}, we investigate a broader class of integrals
that generalizes \cref{prop:values_at_odd_integers} in a natural way. The method employed is
based on contour integration techniques. While the resulting closed-form
expressions may at first appear different, they can in fact be
transformed into one another by means of the functional equations of the
associated \(L\)-functions.
	
\vspace{0.5cm}

\section{\hspace{0.3cm}On the integral $\displaystyle\int_0^{+\infty} \dfrac{\sinh^{2q+1}z\,\ln z}{\cosh^n z}\,dz$}\label{sec:second_integral}

\vspace{0.3cm}

\textit{Blagouchine \textnormal{\cite{Blagouchine2014}} investigated the integral
\[
\int_{0}^{+\infty} \frac{\ln z}{\cosh^{n} z}\,dz
\]
and derived closed-form expressions for several low-order cases. By
means of suitable linearizations, these results extend to integrals of
the form
\[
\int_{0}^{+\infty} \frac{\sinh^{2q} z\,\ln z}{\cosh^{n} z}\,dz.
\]
In the present section, we focus on the complementary family of
integrals
\[
\int_{0}^{+\infty} \frac{\sinh^{2q+1} z\,\ln z}{\cosh^{n} z}\,dz,
\]. Blagouchine \cite{Blagouchine2014} derived the integral
\[
\int_{0}^{+\infty} \frac{\ln z}{\cosh^{n} z}\,dz
\]
from the auxiliary integral
\[
\int_{0}^{+\infty} \frac{\ln (z^2 + a^2)}{\cosh^{n} z}\,dz,
\]
whose integrand is an even function. This symmetry allows one to extend the domain of integration to \((-\infty,\infty)\) and to evaluate the limit as \(a \to 0\).
In our case, however, the integrand
\[
\frac{\sinh^{2q+1} z\, \ln(z^2 + a^2)}{\cosh^{n} z}
\]
is an odd function of the variable \(z\). Consequently, extending the integral to the whole real line,
\[
\int_{-\infty}^{+\infty} \frac{\sinh^{2q+1} z\, \ln(z^2 + a^2)}{\cosh^{n} z}\,dz,
\]
would trivially yield zero. Before proceeding further, we recall that the conditions \(n > 2q + 1\) and \(q \ge 0\) are required to ensure convergence. Nevertheless, we again employ contour integration techniques to evaluate the original integral, after introducing suitable modifications.} 

\vspace{0.3cm}
In particular, upon introducing the substitution \( z = e^{u} \), we obtain

\[
\int_0^{+\infty} \frac{\sinh^{2q+1} z\,\ln z}{\cosh^{n} z}\,dz
=
\int_{-\infty}^{+\infty}
u\,e^{u}\,
\frac{\sinh^{2q+1}(e^{u})}{\cosh^{n}(e^{u})}\,du,
\]

The bounds of integration are now \( -\infty \) and \( +\infty \), which are precisely those required for a standard contour integration. The contour of integration is the rectangle
\[
\Gamma_R := [-R, R] \cup [R, R + i\pi] \cup [R + i\pi, -R + i\pi] \cup [-R + i\pi, -R]
\]

The function being integrated on the contour $\Gamma_R$ is given by 

\[u \to u^2\, f_{q,n}(u) \quad \text{where} \quad f_{q,n}(u) := \frac{e^u\,\sinh^{2q+1}(e^{u})}{\cosh^{n}(e^{u})} \]

Bearing in mind that the contour of integration is counterclockwise, we have the following equality:

\[\ointctrclockwise_{\Gamma_R}z^2 f_{q,n}(z) = \int_{-R}^R z^2 f_{q,n}(z) \, dz + \int_{R}^{R+i \pi} z^2 f_{q,n}(z) \, dz + \int_{R+i\pi}^{-R+i\pi} z^2  f_{q,n}(z)\, dz + \int_{-R+i\pi}^{-R} z^2 f_{q,n}(z) \, dz\]

Performing the substitutions $z = R + ix$, $z = x + i \pi$ and $z = -R + ix$ respectively in the third, fourth and fifth integral yields

\begin{align*}
\ointctrclockwise_{\Gamma_R}z^2 f_{q,n}(z) = &\int_{-R}^R z^2 f_{q,n}(z) \, dz + i\int_{0}^{ \pi} (R + iz)^2 f_{q,n}(R + iz) \, dz - \int_{-R}^{R} (z + i\pi)^2  f_{q,n}(z + i\pi)\, dz\\ &- i\int_{0}^{\pi} (-R + iz)^2 f_{q,n}(-R+iz) \, dz\end{align*}

\vspace{0.3cm}

Now, $f_{q,n}(z + i\pi) = f_{q,n}(z)$. Consequently

\begin{align*}
\ointctrclockwise_{\Gamma_R}z^2 f_{q,n}(z) = &-2i\pi\int_{-R}^R z f_{q,n}(z) \, dz + \pi^2 \int_{-R}^R f_{q,n}(z) \, dz + i\int_{0}^{ \pi} (R + iz)^2 f_{q,n}(R + iz) \, dz \\ &- i\int_{0}^{\pi} (-R + iz)^2 f_{q,n}(-R+iz) \, dz\end{align*}

Under the assumption that $n > 2q+1$, the two last integrals decay as $R$ grows large accordingly to \cref{lemma:vanishing_verticals}. Hence

\begin{equation}
\ointctrclockwise_{\Gamma_\infty}z^2 f_{q,n}(z) = -2i\pi\int_{-\infty}^{+\infty} z f_{q,n}(z) \, dz + \pi^2 \int_{-\infty}^{+\infty} f_{q,n}(z) \, dz \label{eq:infinite_integral}\end{equation}

The poles of the integrand are the solutions of
\[
e^{u} = \pm \frac{(2l+1)i\pi}{2},
\]
which are given by the set
\[
\left\{ \ln \frac{(2l+1)\pi}{2} + i\pi\!\left(2m + \frac{1}{2}\right) \right\}
\;\cup\;
\left\{ \ln \frac{(2l+1)\pi}{2} + i\pi\!\left(2m + \frac{3}{2}\right) \right\},
\qquad l \in \mathbb{N}_0,\; m \in \mathbb{Z}.
\]
Those lying inside the contour of integration \( \Gamma_R \) correspond to \( m = 0 \), yielding the subset
\[
\left\{ \ln \frac{(2l+1)\pi}{2} + \frac{i\pi}{2} \right\},
\qquad l \in \mathbb{N}_0.
\]
Accordingly, these poles are denoted by
\[
u_l := \ln \frac{(2l+1)\pi}{2} + \frac{i\pi}{2},
\qquad l \in \mathbb{N}_0.
\]

Cauchy's residue theorem together with \eqref{eq:infinite_integral} yields

\begin{equation}
\begin{cases}\displaystyle\int_0^{+\infty} \dfrac{\sinh^{2q+1}z\,\ln z}{\cosh^n z}\,dz = -\Re\left\{\sum_{l = 0}^{\infty}\underset{u = u_l}{\text{Res}}\left( \dfrac{u^2 \,e^u\, \sinh^{2q+1}(e^u)}{ \cosh^{n}(e^u)}\right)\right\} \\[0.5cm]
\displaystyle\int_{0}^{\infty}\frac{\sinh^{2q+1}z}{\cosh^{n}z}\,dz
=-\frac{2}{\pi}\,\Im\left\{\sum_{l = 0}^{\infty}\underset{u = u_l}{\text{Res}}\left( \dfrac{u^2 \,e^u\, \sinh^{2q+1}(e^u)}{ \cosh^{n}(e^u)}\right)\right\}
\end{cases}
\label{eq:formula1}
\end{equation}

Notwithstanding, expanding
\[
u^2\,e^{u}\,\frac{\sinh^{2q+1}(e^{u})}{u\,\cosh^{n}(e^{u})}
\]
around \( u_l \) in order to determine the corresponding residue turns out to be a very tedious and cumbersome approach. Instead, we reformulate the definition of the residue using Cauchy's integral formula as follows:

\[\underset{u = u_l}{\text{Res}}\left(\dfrac{u^2 \,e^u\, \sinh^{2q+1}(e^u)}{ \cosh^{n}(e^u)}\right) = \frac{1}{2i\pi}\oint_{\left|u-u_l\right|=\epsilon}\dfrac{u^2 \,e^u\, \sinh^{2q+1}(e^u)}{ \cosh^{n}(e^u)}\,du\]

The change of variable $u = \ln z$ yields

\[\oint_{\left|u-u_l\right|=\epsilon}\dfrac{u^2 \,e^u\, \sinh^{2q+1}(e^u)}{ \cosh^{n}(e^u)}\,du=\oint_{\left|z-z_l\right|=\epsilon}\frac{\ln^2 z\sinh^{2q+1}(z)}{\cosh^n (z)}\,dz = 2i\pi\underset{z = z_l}{\text{Res}}\left(\frac{\ln^2 z \,\sinh^{2q+1}(z)}{\, \cosh^n (z)}\right)\]
$\quad \text{where} \quad \ln z_l := u_l$. And by replacing in \eqref{eq:formula1}, we obtain

\begin{equation}
\begin{cases}\displaystyle\int_0^{+\infty} \dfrac{\sinh^{2q+1}z\,\ln z}{\cosh^n z}\,dz = -\Re\left\{\sum_{l = 0}^{\infty}\underset{z = z_l}{\text{Res}}\left(\frac{\ln^2 z \,\sinh^{2q+1}(z)}{\, \cosh^n (z)}\right)\right\} \\[0.5cm]
\displaystyle\int_{0}^{\infty}\frac{\sinh^{2q+1}z}{\cosh^{n}z}\,dz
=
-\frac{2}{\pi}\,
\Im\left\{
\sum_{l = 0}^{\infty}\underset{z = z_l}{\text{Res}}\left(\frac{\ln^2 z \,\sinh^{2q+1}(z)}{\, \cosh^n (z)}\right)\right\}
\end{cases}
\label{eq:formula2}
\end{equation}

\vspace{0.3cm}

By applying results from \cite[35--48]{talla_waffo_integral_2025}, we obtain the following general formulae:

\begin{equation}
\begin{cases}
\underset{z = z_l}{\text{Res}}\left( \dfrac{L(z)}{ \cosh^{2n+1}z}\right)  = \dfrac{1}{\sinh^{2n+1}z_l}\displaystyle\sum_{m=0}^n c_{2n-2m,n} \dfrac{L^{(2m)}(z_l)}{(2m)!} \\
\vspace{0.3cm}\\
\underset{z = z_l}{\text{Res}}\left( \dfrac{L(z)}{\cosh^{2n}z}\right) = \dfrac{1}{\sinh^{2n}z_l}\displaystyle\sum_{m=0}^{n-1} d_{2n-2m-2,n} \dfrac{L^{(2m+1)}(z_l)}{(2m+1)!}
\end{cases} 
\label{eq:residues}
\end{equation} And applying for $L_q(z) := \ln^2 z \,\,\sinh^{2q+1}z$, one has :

\[\begin{cases}
L_q^{(2m+1)}(z_l) = \left. \displaystyle\sum_{p = 0}^{2m+1} \binom{2m+1}{p} \dfrac{d^{2m+1-p}}{dz^{2m+1-p}}\left(\ln^2 z\right) \dfrac{d^{p}}{dz^{p}}\left(\sinh^{2q+1}z\right) \right|_{z = z_l} \\
\vspace{0.3cm}\\
L_q^{(2m)}(z_l) = \left. \displaystyle\sum_{p = 0}^{2m} \binom{2m}{p} \dfrac{d^{2m-p}}{dz^{2m-p}}\left(\ln^2 z\right) \dfrac{d^{p}}{dz^{p}}\left(\sinh^{2q+1}z\right) \right|_{z = z_l}
\end{cases}\]

Keeping in mind that \cite[18]{talla_waffo_integral_2025}

\[\sinh^{2q+1} z
= \frac{1}{4^{q}}
\sum_{k=0}^{q}
(-1)^k \binom{2q+1}{k}\,
\sinh\!\bigl((2q+1-2k)z\bigr)
\], it follows

\[\frac{d^p}{dz^p}\left(\sinh^{2q+1} z\right)
= \frac{1}{4^{q}}
\sum_{k=0}^{q}
(-1)^k \binom{2q+1}{k}\,
(2q+1-2k)^p h_{q,p,k}(z)
\] \text{where} \[\quad h_{q,p,k}(z) := \begin{cases}\cosh\!\bigl((2q+1-2k)z\bigr) \quad \text{if} \quad p \quad \text{is odd}\\
\sinh\!\bigl((2q+1-2k)z\bigr) \quad \text{if} \quad p \quad \text{is even}\end{cases}\]

A simple evaluation shows, \[h_{q,p,k}(z_l) := \begin{cases}0 &\quad \text{if} \quad p \quad \text{is odd}\\
(-1)^{l+q-k}i &\quad \text{if} \quad p \quad \text{is even}\end{cases}\]

It comes out,

\vspace{1cm}

\[\begin{cases}
L_q^{(2m+1)}(z_l) = \left. \displaystyle\sum_{p = 0}^{m} \binom{2m+1}{2p} \dfrac{d^{2m+1-2p}}{dz^{2m+1-2p}}\left(\ln^2 z\right) \dfrac{d^{2p}}{dz^{2p}}\left(\sinh^{2q+1}z\right) \right|_{z = z_l} \\
\vspace{0.3cm}\\
L_q^{(2m)}(z_l) = \left. \displaystyle\sum_{p = 0}^{m} \binom{2m}{2p} \dfrac{d^{2m-2p}}{dz^{2m-2p}}\left(\ln^2 z\right) \dfrac{d^{2p}}{dz^{2p}}\left(\sinh^{2q+1}z\right) \right|_{z = z_l}
\end{cases}\]

This simplifies to

\[\begin{cases}
L_q^{(2m+1)}(z_l) = (-1)^{l+q}\,i\left. \displaystyle\sum_{p = 0}^{m} \binom{2m+1}{2m-2p} \Omega_{q,m-p} \dfrac{d^{2p+1}}{dz^{2p+1}}\left(\ln^2 z\right) \right|_{z = z_l} \\
\vspace{0.3cm}\\
L_q^{(2m)}(z_l) = (-1)^{l+q}\,i\left. \displaystyle\sum_{p = 0}^{m} \binom{2m}{2m-2p} \Omega_{q,m-p} \dfrac{d^{2p}}{dz^{2p}}\left(\ln^2 z\right) \right|_{z = z_l}
\end{cases}\]

where \begin{equation}\Omega_{q,p} :=\frac{1}{(-1)^{l+q}\,i}\left.\dfrac{d^{2p}}{dz^{2p}}\left(\sinh^{2q+1}z\right) \right|_{z = z_l} = \frac{1}{4^{q}}
\sum_{k=0}^{q} \binom{2q+1}{k}\,
(2q+1-2k)^{2p}\label{eq:definition_of_theta}\end{equation}

Additionally, for $n$ being nonzero, 

\[\frac{d^n}{d^n}\left(\ln^2 z\right) = \frac{2(-1)^{n-1}(n-1)!}{z^n}\left(\ln z - H_{n-1}\right)\]
where $H_{n}$ stands for the $n$-th harmonic number with the special condition $H_0 = 0$

Hence,

\[\begin{cases}
L_q^{(2m+1)}(z_l) = (-1)^{l+q}\,i \displaystyle\sum_{p = 0}^{m} \binom{2m+1}{2m-2p} \Omega_{q,m-p} \frac{2(2p)!}{z_l^{2p+1}}\left(\ln z_l - H_{2p}\right) \\
\vspace{0.3cm}\\
L_q^{(2m)}(z_l) = (-1)^{l+q}\,i\left\{\Omega_{q,m}\, \ln^2 (z_l) -  \displaystyle\sum_{p = 1}^{m} \binom{2m}{2m-2p} \Omega_{q,m-p} \frac{2(2p-1)!}{z_l^{2p}}\left(\ln z_l - H_{2p-1}\right) \right\}
\end{cases}\]

Using \eqref{eq:residues}, we conclude

\[
\begin{cases}
\underset{z = z_l}{\text{Res}}\left( \dfrac{\ln^2 z \, \sinh^{2q+1}z}{ \cosh^{2n+1}z}\right)  = (-1)^{q+n}\displaystyle\sum_{m=0}^n  \dfrac{c_{2n-2m,n}}{(2m)!}\left\{\Omega_{q,m}\, \ln^2 (z_l)\right. \\
\hspace{9cm}- \left. \displaystyle\sum_{p = 1}^{m} \binom{2m}{2m-2p} \Omega_{q,m-p} \frac{2(2p-1)!}{z_l^{2p}}\left(\ln z_l - H_{2p-1}\right) \right\} \\
\vspace{0.3cm}\\
\underset{z = z_l}{\text{Res}}\left( \dfrac{\ln^2 z \, \sinh^{2q+1}z}{\cosh^{2n}z}\right) = (-1)^{l+q+n}\,i\displaystyle\sum_{m=0}^{n-1}  \dfrac{d_{2n-2m-2,n}}{(2m+1)!}\displaystyle\sum_{p = 0}^{m} \binom{2m+1}{2m-2p} \Omega_{q,m-p} \frac{2(2p)!}{z_l^{2p+1}}\left(\ln z_l - H_{2p}\right)
\end{cases} \]

Interestingly, this equality holds for $q$ and $n$ naturals satisfying $0 \leq q < n$ (See the proof at \cref{prop:vanishing}): \begin{equation*}\displaystyle\sum_{m=0}^n  \dfrac{c_{2n-2m,n}}{(2m)!} \Omega_{q,m} = 0\end{equation*} Thus

\begin{equation}
\begin{cases}
\underset{z = z_l}{\text{Res}}\left( \dfrac{\ln^2 z \, \sinh^{2q+1}z}{ \cosh^{2n+1}z}\right)  = (-1)^{q+n+1} \displaystyle\sum_{m=0}^n  \dfrac{c_{2n-2m,n}}{(2m)!}\displaystyle\sum_{p = 1}^{m} \binom{2m}{2m-2p} \Omega_{q,m-p} \frac{2(2p-1)!}{z_l^{2p}}\left(\ln z_l - H_{2p-1}\right) \\
\vspace{0.3cm}\\
\underset{z = z_l}{\text{Res}}\left( \dfrac{\ln^2 z \, \sinh^{2q+1}z}{\cosh^{2n}z}\right) = (-1)^{l+q+n}\,i\displaystyle\sum_{m=0}^{n-1}  \dfrac{d_{2n-2m-2,n}}{(2m+1)!}\displaystyle\sum_{p = 0}^{m} \binom{2m+1}{2m-2p} \Omega_{q,m-p} \frac{2(2p)!}{z_l^{2p+1}}\left(\ln z_l - H_{2p}\right)
\end{cases} 
\label{eq:formulae_of_residue}
\end{equation}

This finally leads to these results after using \eqref{eq:formula2}

\begin{equation}
\begin{cases}
\displaystyle\int_0^{+\infty} \dfrac{\sinh^{2q+1}z\,\ln z}{\cosh^{2n+1} z}\,dz = \displaystyle\sum_{p = 0}^{n-1}H_{p,q,n} \frac{\zeta'(2p+2)}{\pi^{2p+2}} + I_{q,n} + J_{q,n}\ln \pi + (K_{q,n} - J_{q,n})\ln 2 \\
\vspace{0.3cm}\\
\displaystyle\int_0^{+\infty} \dfrac{\sinh^{2q+1}z\,\ln z}{\cosh^{2n} z}\,dz = \displaystyle\sum_{p = 0}^{n-1}L_{p,q,n} \frac{\beta'(2p+1)}{\pi^{2p+1}} + M_{q,n} + N_{q,n}\ln \pi - N_{q,n}\ln 2
\end{cases}
\label{eq:main_formulae}
\end{equation}

\vspace{0.5cm}

where

\[S_{p,q,n}
:=
\sum_{m=p+1}^{n}
\frac{c_{2n-2m,n}}{(2m)!}
\binom{2m}{2p+2}
\Omega_{q,m-p-1}
\]
\begin{equation}H_{p,q,n}
:=
(-1)^{q+n+p}\;
2(2p+1)!\,
\bigl(2^{2p+2}-1\bigr)\;
S_{p,q,n}
\label{eq:formula_of_h}
\end{equation}
\[J_{q,n}
:=
(-1)^{q+n+1}
\sum_{p=0}^{n-1}
\frac{2^{2p+1}\bigl(2^{2p+2}-1\bigr)}{p+1}\,
B_{2p+2}\,
S_{p,q,n}
\]
\[K_{q,n}
:=
(-1)^{q+n}
\sum_{p=0}^{n-1}
\frac{2^{2p+1}}{p+1}\,
B_{2p+2}\,
S_{p,q,n}
\]
\[I_{q,n}
:=
(-1)^{q+n}
\sum_{p=0}^{n-1}
\frac{2^{2p+1}\bigl(2^{2p+2}-1\bigr)}{p+1}\,
B_{2p+2}\,
H_{2p+1}\,
S_{p,q,n}
\]
\[
L_{p,q,n}
=
(-1)^{q+n}\,2^{2p+2}(2p)!\,(-1)^p
\sum_{m=p}^{n-1}\frac{d_{2n-2m-2,n}}{(2m+1)!}
\binom{2m+1}{2m-2p}\,\Omega_{q,m-p},
\]

\[
N_{q,n}
=
(-1)^{q+n+1}
\sum_{m=0}^{n-1}\frac{d_{2n-2m-2,n}}{(2m+1)!}
\sum_{p=0}^{m}\binom{2m+1}{2m-2p}\,
\Omega_{q,m-p}\,
E_{2p}
\]
\[
M_{q,n}
=
(-1)^{q+n}
\sum_{m=0}^{n-1}\frac{d_{2n-2m-2,n}}{(2m+1)!}
\sum_{p=0}^{m}\binom{2m+1}{2m-2p}\,
\Omega_{q,m-p}\,
\bigl(H_{2p}\,E_{2p}\bigr)
\]
where $B_{2p}$ and $E_{2p}$ are Bernoulli and Euler numbers with classic conventions.

\vspace{0.3cm}

Examples

\[
\int_{0}^{\infty}
\frac{\sinh z\,\ln z}{\cosh^{3} z}\,dz
=
-3\,\frac{\zeta'(2)}{\pi^{2}}
-\frac{1}{2}+
\frac{1}{2}\ln \pi - \frac{2}{3}\ln 2
\]
\vspace{0.3cm}

\[
\int_{0}^{\infty}
\frac{\sinh z\,\ln z}{\cosh^{5} z}\,dz
=
-\frac{\zeta'(2)}{\pi^{2}}
-\frac{15}{2}\,\frac{\zeta'(4)}{\pi^{4}}
-\frac{23}{72}+
\frac{1}{4}\ln \pi - \frac{14}{45}\ln 2
\]
\vspace{0.3cm}
\[
\int_{0}^{\infty}
\frac{\sinh^{3} z\,\ln z}{\cosh^{5} z}\,dz
=
-2\,\frac{\zeta'(2)}{\pi^{2}}
+\frac{15}{2}\,\frac{\zeta'(4)}{\pi^{4}}
-\frac{13}{72}
+\frac{1}{4}\ln \pi - \frac{16}{45}\ln 2
\]
\vspace{0.3cm}
\[
\int_{0}^{\infty}
\frac{\sinh z\,\ln z}{\cosh^{7} z}\,dz
=
-\frac{8}{15}\,\frac{\zeta'(2)}{\pi^{2}}
-5\,\frac{\zeta'(4)}{\pi^{4}}
-21\,\frac{\zeta'(6)}{\pi^{6}}
-\frac{163}{675}
+\frac{1}{6}\ln \pi - \frac{568}{2\,835}\ln 2
\]
\vspace{0.3cm}
\[
\int_{0}^{\infty}
\frac{\sinh^{3} z\,\ln z}{\cosh^{7} z}\,dz
=
-\frac{7}{15}\,\frac{\zeta'(2)}{\pi^{2}}
-\frac{5}{2}\,\frac{\zeta'(4)}{\pi^{4}}
+21\,\frac{\zeta'(6)}{\pi^{6}}
-\frac{421}{5\,400}
+\frac{1}{12}\ln \pi - \frac{314}{2\,835}\ln 2
\]
\vspace{0.3cm}
\[
\int_{0}^{\infty}
\frac{\sinh^{5} z\,\ln z}{\cosh^{7} z}\,dz
=
-\frac{23}{15}\,\frac{\zeta'(2)}{\pi^{2}}
+10\,\frac{\zeta'(4)}{\pi^{4}}
-21\,\frac{\zeta'(6)}{\pi^{6}}
-\frac{277}{2\,700}
+\frac{1}{6}\ln \pi - \frac{694}{2\,835}\ln 2
\]

\vspace{1cm}

\[
\int_0^{\infty} \frac{\sinh z\,\ln z}{\cosh^{2} z}\,dz
=
-4\,\frac{\beta'(1)}{\pi}
+\ln \pi
-\ln 2
\]
\vspace{0.3cm}
\[
\int_0^{\infty} \frac{\sinh z\,\ln z}{\cosh^{4} z}\,dz
=
-\frac{2}{3}\,\frac{\beta'(1)}{\pi}
-\frac{16}{3}\,\frac{\beta'(3)}{\pi^{3}}
-\frac{1}{4}
+\frac{1}{3}\ln \pi
-\frac{1}{3}\ln 2
\]
\vspace{0.3cm}
\[
\int_0^{\infty} \frac{\sinh^{3} z\,\ln z}{\cosh^{4} z}\,dz
=
-\frac{10}{3}\,\frac{\beta'(1)}{\pi}
+\frac{16}{3}\,\frac{\beta'(3)}{\pi^{3}}
+\frac{1}{4}
+\frac{2}{3}\ln \pi
-\frac{2}{3}\ln 2
\]
\vspace{0.3cm}
\[
\int_0^{\infty} \frac{\sinh z\,\ln z}{\cosh^{6} z}\,dz
=
-\frac{3}{10}\,\frac{\beta'(1)}{\pi}
-\frac{8}{3}\,\frac{\beta'(3)}{\pi^{3}}
-\frac{64}{5}\,\frac{\beta'(5)}{\pi^{5}}
-\frac{61}{288}
+\frac{1}{5}\ln \pi
-\frac{1}{5}\ln 2
\]
\vspace{0.3cm}
\[
\int_0^{\infty} \frac{\sinh^{3} z\,\ln z}{\cosh^{6} z}\,dz
=
-\frac{11}{30}\,\frac{\beta'(1)}{\pi}
-\frac{8}{3}\,\frac{\beta'(3)}{\pi^{3}}
+\frac{64}{5}\,\frac{\beta'(5)}{\pi^{5}}
-\frac{11}{288}
+\frac{2}{15}\ln \pi
-\frac{2}{15}\ln 2
\]
\vspace{0.3cm}
\[
\int_0^{\infty} \frac{\sinh^{5} z\,\ln z}{\cosh^{6} z}\,dz
=
-\frac{89}{30}\,\frac{\beta'(1)}{\pi}
+8\,\frac{\beta'(3)}{\pi^{3}}
-\frac{64}{5}\,\frac{\beta'(5)}{\pi^{5}}
+\frac{83}{288}
+\frac{8}{15}\ln \pi
-\frac{8}{15}\ln 2
\]

\vspace{0.3cm}
\begin{lemma}\label[lemma]{lemma:vanishing_verticals}
\[
\lim_{R\to\infty}
\int_{R}^{R+i\pi} u^2 f_{q,n}(u)\,du = 0,
\qquad
\lim_{R\to\infty}
\int_{-R+i\pi}^{-R} u^2 f_{q,n}(u)\,du = 0.
\] where \[f_{q,n}(u) := \frac{e^u\,\sinh^{2q+1}(e^{u})}{\cosh^{n}(e^{u})}\]
\end{lemma}

\begin{proof}
Write $f(u)=f_{q,n}(u)$.

\medskip
\noindent\textbf{Left vertical side.}
Parametrize the left side by $u=-R+it$, $t\in[0,\pi]$. Then $e^{u}=e^{-R}e^{it}$, hence
$|e^{u}|=e^{-R}\to0$ as $R\to\infty$. Using the Taylor expansions
\[
\sinh z = z+O(z^3), \qquad \cosh z = 1+O(z^2)\qquad (z\to0),
\]
we obtain uniformly in $t\in[0,\pi]$,
\[
\frac{\sinh^{2q+1}(e^{u})}{\cosh^{n}(e^{u})}
= (e^{u})^{2q+1}\bigl(1+O(e^{2u})\bigr).
\]
Therefore
\[
|f(-R+it)|
= \left|e^{u}\frac{\sinh^{2q+1}(e^{u})}{\cosh^{n}(e^{u})}\right|
\ll |e^{u}|\cdot |e^{u}|^{2q+1}
= |e^{u}|^{2q+2}
= e^{-(2q+2)R},
\]
uniformly for $t\in[0,\pi]$. Hence
\[
\left|\int_{-R+i\pi}^{-R} u^2 f(u)\,du\right|
= \left|\int_{0}^{\pi}(-R+it)^2 f(-R+it)\,i\,dt\right|
\le \int_{0}^{\pi} |{-R+it}|^2\,|f(-R+it)|\,dt
\]
\[
\ll (R^2+\pi^2)\int_{0}^{\pi} e^{-(2q+2)R}\,dt
\ll (R^2+\pi^2)\,\pi\,e^{-(2q+2)R}\xrightarrow[R\to\infty]{}0.
\]

\medskip
\noindent\textbf{Right vertical side.}
Parametrize the right side by $u=R+it$, $t\in[0,\pi]$, and set $w=e^{u}=e^{R}e^{it}$.
Write $m:=n-(2q+1)>0$. Using $\sinh^{2q+1}(w)=\tanh^{2q+1}(w)\cosh^{2q+1}(w)$ we get
\[
f(R+it)=e^{R+it}\,\tanh^{2q+1}(w)\,\cosh^{-m}(w),
\qquad m>0.
\]

Fix $\delta\in(0,\pi/2)$. On the set
\[
E_\delta := [0,\tfrac{\pi}{2}-\delta]\cup[\tfrac{\pi}{2}+\delta,\pi]
\]
we have $|\cos t|\ge \sin\delta$, hence
\[
|\Re(w)|=|e^{R}\cos t|\ge e^{R}\sin\delta.
\]
Using $\cosh w = \frac12(e^{w}+e^{-w})$ and the reverse triangle inequality,
\[
|\cosh w|
\ge \frac12\bigl||e^{w}|-|e^{-w}|\bigr|
= \frac12\bigl|e^{\Re(w)}-e^{-\Re(w)}\bigr|
= \sinh(|\Re(w)|),
\]
so for $t\in E_\delta$,
\[
|\cosh(e^{R+it})|
=|\cosh w|
\ge \sinh(e^{R}\sin\delta)
\gg \exp(e^{R}\sin\delta).
\]
Moreover, $\tanh(w)\to\pm1$ as $\Re(w)\to\pm\infty$, hence $|\tanh(w)|$ is bounded on
$\{w=e^{R}e^{it}:t\in E_\delta\}$ for all large $R$. Therefore
\[
|f(R+it)|
\ll e^{R}\,|\cosh w|^{-m}
\ll e^{R}\exp\!\bigl(-m\,e^{R}\sin\delta\bigr),
\qquad t\in E_\delta,
\]
and consequently
\[
\int_{E_\delta} |(R+it)^2 f(R+it)|\,dt
\ll (R^2+\pi^2)\,e^{R}\exp\!\bigl(-m\,e^{R}\sin\delta\bigr)
\xrightarrow[R\to\infty]{}0.
\]

It remains to control a small neighborhood of $t=\pi/2$.
Let $J_\eta:=[\frac{\pi}{2}-\eta,\frac{\pi}{2}+\eta]$ with $\eta>0$.
By assumption, the contour does not pass through poles of $f$, i.e. $\cosh(e^{R+it})\neq0$
for $t\in[0,\pi]$ and all sufficiently large $R$. Hence, for each fixed $\eta>0$ there exists
$c_\eta>0$ and $R_\eta>0$ such that
\[
\inf_{t\in J_\eta}|\cosh(e^{R+it})|\ge c_\eta
\qquad (R\ge R_\eta).
\]
Using $|\sinh z|\le |\cosh z|$ for all $z\in\mathbb{C}$, we obtain for $t\in J_\eta$,
\[
|f(R+it)|
=\left|e^{u}\frac{\sinh^{2q+1}(e^{u})}{\cosh^{n}(e^{u})}\right|
\le e^{R}\frac{|\cosh(e^{u})|^{2q+1}}{|\cosh(e^{u})|^{n}}
= e^{R}|\cosh(e^{u})|^{-m}
\le e^{R}c_\eta^{-m}.
\]
Therefore
\[
\int_{J_\eta} |(R+it)^2 f(R+it)|\,dt
\ll (R^2+\pi^2)\,e^{R}\,|J_\eta|
= 2\eta\,(R^2+\pi^2)\,e^{R}.
\]
Choose $\eta=\eta(R):=e^{-(1+\varepsilon)R}$ with any $\varepsilon>0$. Then $\eta(R)\to0$ and
\[
\int_{J_{\eta(R)}} |(R+it)^2 f(R+it)|\,dt
\ll (R^2+\pi^2)\,e^{R}\,e^{-(1+\varepsilon)R}
= (R^2+\pi^2)\,e^{-\varepsilon R}\xrightarrow[R\to\infty]{}0.
\]

Combining the estimates on $E_{\eta(R)}$ and $J_{\eta(R)}$ yields
\[
\lim_{R\to\infty}\int_{R}^{R+i\pi} u^2 f(u)\,du = 0.
\]
Together with the left-side estimate, this completes the proof.
\end{proof}

\vspace{0.3cm}

\begin{lemma}\label[lemma]{lemma:from_sum_to_integral}
\[\forall q, n, r \in \mathbb{N}_0, \,\,\sum_{m=0}^n \dfrac{c_{2n-2m,r}}{(2m)!} \Omega_{q,m} = (-1)^{\,r-q}\frac{1}{2\pi i}
\oint_{|x|=\varepsilon}
\bigl(\operatorname{arctanh}x\bigr)^{2r-2n}
\frac{(1-x^2)^{\,r-q-1}}{x^{2r+1}}\,dx\] where $\varepsilon$ is a positive real number so small enough that the contour encloses only the singularity $0$
\end{lemma} 

\begin{proof}

\vspace {0.3cm}

Let $R_{q, n, r} :=\displaystyle\sum_{m=0}^n \dfrac{c_{2n-2m,r}}{(2m)!} \Omega_{q,m}$. Bearing in mind that \[\Omega_{q,p} := \frac{1}{(-1)^{l+q}\,i}\left.\dfrac{d^{2p}}{dz^{2p}}\left(\sinh^{2q+1}z\right) \right|_{z = z_l} \] as defined by \eqref{eq:definition_of_theta}, we have

\[R_{q, n, r} =\frac{1}{(-1)^{l+q}\,i}\displaystyle\sum_{m=0}^n \dfrac{c_{2n-2m,r}}{(2m)!} \left.\dfrac{d^{2m}}{dz^{2m}}\left(\sinh^{2q+1}z\right) \right|_{z = z_l}\]

Since the coefficients $c_{j,n}$ are zero and $\left.\dfrac{d^{j}}{dz^{j}}\left(\sinh^{2q+1}z\right) \right|_{z = z_l} = 0$ if $j$ is odd, the sum may be understood as running up to $2n$ as follows

\[R_{q, n, r} =\frac{1}{(-1)^{l+q}\,i}\displaystyle\sum_{m=0}^{2n} \dfrac{c_{2n-m,r}}{m!} \left.\dfrac{d^{m}}{dz^{m}}\left(\sinh^{2q+1}z\right) \right|_{z = z_l}\]

$R_{q, n}$ may be interpreted as a sum of this kind 

\[R_{q, n, r} = \frac{1}{(-1)^{l+q}\,i}\sum_{m = 0}^{2n} c_{2n-m,r} \, t_{m,q} \quad \text{where} \quad t_{m,q} := \frac{1}{m!}\left. \dfrac{d^{m}}{dz^{m}}\left(\sinh^{2q+1}z\right) \right|_{z = z_l}\]

The latter sum is clearly recognized as a coefficient appearing in the Cauchy's product formula of two series. Thus

\[R_{q, n, r} =\frac{1}{(-1)^{l+q}\,i} \left[(z - z_l)^{2n}\right] C_r(z)\, T_q(z) \quad \text{where} \quad \begin{cases}C_r(z) := \displaystyle\sum_{m=0}^{\infty}c_{m,r}\,(z - z_l)^m \\
\vspace{0.1cm}\\
T_q(z) := \displaystyle\sum_{m=0}^{\infty}t_{m,q}\,(z - z_l)^m\end{cases}\]
 
From \cite[39--43]{talla_waffo_integral_2025}, we know

\[\frac{1}{\cosh^{2r+1}z} = \frac{1}{(-1)^{r+l}i \,\, (z-z_l)^{2r+1}}\sum_{m = 0}^{\infty}c_{m,r} (z - z_l)^{m}\]

It follows

\[\begin{cases}C_r(z) = \dfrac{(-1)^{r+l}i \,\, (z-z_l)^{2r+1}}{\cosh^{2r+1}z}\\
\vspace{0.1cm}\\
T_q(z) = \sinh^{2q+1}z\end{cases}\]

Hence
\[R_{q, n, r} = (-1)^{q+r}\left[(z - z_l)^{2n}\right] \frac{(z-z_l)^{2r+1}}{\cosh^{2r+1}z}\,\sinh^{2q+1}z\]

The Cauchy's integral formula yields

\[R_{q, n, r} =  (-1)^{q+r}\frac{1}{2i\pi}\oint_{\left|z-z_l\right|= \varepsilon} (z-z_l)^{2r-2n}\frac{\sinh^{2q+1}z}{\cosh^{2r+1}z}\]

where $\varepsilon$ is a positive real number so small enough that the contour encloses only the singularity $z_l$. Since \[\cosh(z_l + w) = \sinh z_l \, \sinh w \qquad \text{and} \qquad \sinh(z_l + w) = \sinh z_l \, \cosh w\], the change of variable $w = z - z_l$ yields

\[
R_{q,n,r}=(-1)^{\,r-q}\frac{1}{2\pi i}\oint_{|w|=\varepsilon}
w^{\,2r-2n}\,\frac{\cosh^{2q+1}w}{\sinh^{2n+1}w}\,dw
\]

where $\varepsilon$ is a positive real number so small enough that the contour encloses only the singularity $0$. The change of variable $w = \text{arctanh} x$ with \[\sinh(\text{arctanh} x) = \dfrac{x}{\sqrt{1-x^2}}\qquad \text{and} \qquad \cosh(\text{arctanh} x) = \dfrac{1}{\sqrt{1-x^2}}\] lead to the integral of the claim. This concludes the proof.
\end{proof}

\vspace{0.5cm}

\begin{proposition}\label[proposition]{prop:vanishing}
\[\forall q, n \in \mathbb{N} \qquad (0 \leq q < n), \sum_{m=0}^n \dfrac{c_{2n-2m,n}}{(2m)!} \Omega_{q,m} = 0\]
\end{proposition}

\begin{proof}

\vspace {0.3cm}

The exercise consists of proving that $R_{q, n} :=\displaystyle\sum_{m=0}^n \dfrac{c_{2n-2m,n}}{(2m)!} \Omega_{q,m} = 0$ if $q < n$ holds. \Cref{lemma:from_sum_to_integral} leads to

\[
R_{q,n}
=(-1)^{q+n}
\frac{1}{2i\pi}\oint_{|x|=\varepsilon}
\frac{(1-x^2)^{\,n-q-1}}{x^{2n+1}}\,dx.
\]

Since \(q<n\), the exponent \(n-q-1\) is a non-negative integer. Hence the factor
\((1-x^2)^{n-q-1}\) can be expanded as a finite binomial series,
\[
(1-x^2)^{n-q-1}
=
\sum_{j=0}^{n-q-1}
\binom{n-q-1}{j}(-1)^j x^{2j}.
\]
Substituting this expansion into the integrand gives
\[
\frac{(1-x^2)^{n-q-1}}{x^{2n+1}}
=
\sum_{j=0}^{n-q-1}
\binom{n-q-1}{j}(-1)^j\,x^{2j-2n-1}.
\]

The residue at \(x=0\) is the coefficient of \(x^{-1}\) in this Laurent series.
Such a term would require
\[
2j-2n-1=-1 \quad\Longleftrightarrow\quad j=n.
\]
However, the summation index satisfies \(j\le n-q-1<n\), so this condition can
never be met. Therefore the coefficient of \(x^{-1}\) vanishes, the residue is
zero, and consequently
\[
R_{q,n}:=\displaystyle\sum_{m=0}^n \dfrac{c_{2n-2m,n}}{(2m)!} \Omega_{q,m}=0 \qquad (q<n).
\]
This concludes the proof.
\end{proof}

\vspace{0.5cm}

\begin{lemma}\label[lemma]{lemma:euler_and_bernoulli}
\[\forall q, n : 0 \leq q < n\]\[
\begin{cases}
\displaystyle
\sum_{m=0}^{n}\frac{c_{2n-2m,n}}{(2m)!}
\sum_{p=1}^{m}\binom{2m}{2m-2p}\,
\Omega_{q,m-p}\,
\frac{2^{2p-1}(2^{2p}-1)}{p}\,B_{2p}
=
(-1)^{q+n+1}\,
\frac12\,\frac{q!\,(n-q-1)!}{n!}
\\
\vspace{0.3cm}\\
\displaystyle
\sum_{m=0}^{n-1}\frac{d_{2n-2m-2,n}}{(2m+1)!}
\sum_{p=0}^{m}\binom{2m+1}{2m-2p}\,
\Omega_{q,m-p}\,
E_{2p}
=
(-1)^{q+n+1}\,
2^{2q+1}\,q!\,
\frac{n!\,(2n-2q-2)!}{(n-q-1)!\,(2n)!}
\end{cases}
\]
\end{lemma}

\begin{proof}

Combining \eqref{eq:formulae_of_residue} and \eqref{eq:formula2} yields

\[
\int_{0}^{\infty}\frac{\sinh^{2q+1}z}{\cosh^{N}z}\,dz
=
\begin{cases}
\displaystyle
(-1)^{q+n+1}
\sum_{m=0}^{n}\frac{c_{2n-2m,n}}{(2m)!}
\sum_{p=1}^{m}\binom{2m}{2m-2p}\,
\Omega_{q,m-p}\,
\frac{2^{2p-1}(2^{2p}-1)}{p}\,B_{2p},
& N=2n+1,
\\
\vspace{0.3cm}\\
\displaystyle
(-1)^{q+n+1}
\sum_{m=0}^{n-1}\frac{d_{2n-2m-2,n}}{(2m+1)!}
\sum_{p=0}^{m}\binom{2m+1}{2m-2p}\,
\Omega_{q,m-p}\,
E_{2p},
& N=2n.
\end{cases}
\]

The left-hand side integrals can be evaluated through Euler's beta and gamma functions by performing the substitution $t = \tanh z$. This ends the proof.
\end{proof}

\vspace{0.5cm}

\section{\hspace{0.3cm}On the integral $\displaystyle\int_0^{+\infty} \dfrac{\sinh^{2q}z}{z \, \cosh^{n}z}\,dz$}\label{sec:first_integral}

\vspace{0.5cm}

The integral
\[
\int_0^{+\infty} \frac{\sinh^{2q} z}{z\,\cosh^{n} z}\,dz
\]
in this form does not satisfy all the requirements for a direct evaluation via the contour integration technique used in \cite{talla_waffo_integral_2025}. We note in passing that this integral corresponds to the case
\(m \not\equiv n \pmod{2}\) of the more general integral
\[
\int_{0}^{+\infty} \frac{\sinh^{n}(p z)}{z^{m}\,\cosh^{q} z}\,dz,
\]
which was omitted in the previous work for the sake of brevity; see
\cite[Sec.~5, par.~4, pp.~114--115]{talla_waffo_integral_2025}.
The lower-order cases can be distinguished by successive integrations by
parts combined with suitable linearizations of the numerator. 

The condition of convergence requires $0 < 2q < n$. We study this class of integrals with only one step. After performing an integration by parts with $u'() = 1/z$ such
that $u(z) = \ln z$, we arrive at the conclusion that\[
\int_{0}^{\infty}\frac{\sinh^{2q}z}{z\,\cosh^{n}z}\,dz
=\Big[\ln z\,\frac{\sinh^{2q}z}{\cosh^{n}z}\Big]_{0}^{\infty}
-\int_{0}^{\infty}\ln z\,\frac{\sinh^{2q}z}{\cosh^{n}z}
\big(2q\,\coth z-n\,\tanh z\big)\,dz
\]

We can even show that boundary terms vanish by evaluating the limits. Hence,
\begin{equation}
\int_{0}^{\infty}\frac{\sinh^{2q}z}{z\,\cosh^{n}z}\,dz
=-2q\int_{0}^{\infty}\ln z\,\frac{\sinh^{2q-1}z}{\cosh^{n-1}z}\,dz
+n\int_{0}^{\infty}\ln z\,\frac{\sinh^{2q+1}z}{\cosh^{n+1}z}\,dz
\label{eq:bridge_equality}
\end{equation}

\vspace{0.3cm}

By using previous results of \cref{sec:second_integral}, we derive these examples

\vspace{0.3cm}

\[
\int_{0}^{\infty}\frac{\sinh^{2}z}{z\,\cosh^{4}z}\,dz
= \Phi_1(3) = 
-2\frac{\zeta'(2)}{\pi^{2}}
+30\frac{\zeta'(4)}{\pi^{4}}
+\frac{5}{18}
-\frac{4}{45}\ln 2
\]

\vspace{0.3cm}
\[
\int_{0}^{\infty}\frac{\sinh^{2}z}{z\,\cosh^{6}z}\,dz
=
-\frac{4}{5}\frac{\zeta'(2)}{\pi^{2}}
+126\frac{\zeta'(6)}{\pi^{6}}
+\frac{77}{450}
-\frac{8}{189}\ln 2
\]
\vspace{0.3cm}
\[
\int_{0}^{\infty}\frac{\sinh^{4}z}{z\,\cosh^{6}z}\,dz
=\Phi_1(5) =
-\frac{6}{5}\frac{\zeta'(2)}{\pi^{2}}
+30\frac{\zeta'(4)}{\pi^{4}}
-126\frac{\zeta'(6)}{\pi^{6}}
+\frac{8}{75}
-\frac{44}{945}\ln 2
\]

\vspace{0.3cm}
\[
\int_{0}^{\infty}\frac{\sinh^{2}z}{z\,\cosh^{8}z}\,dz
=
-\frac{16}{35}\frac{\zeta'(2)}{\pi^{2}}
-2\frac{\zeta'(4)}{\pi^{4}}
+42\frac{\zeta'(6)}{\pi^{6}}
+510\frac{\zeta'(8)}{\pi^{8}}
+\frac{16\,469}{132\,300}
-\frac{368}{14\,175}\ln 2
\]
\vspace{0.3cm}
\[
\int_{0}^{\infty}\frac{\sinh^{4}z}{z\,\cosh^{8}z}\,dz
=
-\frac{12}{35}\frac{\zeta'(2)}{\pi^{2}}
+2\frac{\zeta'(4)}{\pi^{4}}
+84\frac{\zeta'(6)}{\pi^{6}}
-510\frac{\zeta'(8)}{\pi^{8}}
+\frac{6\,169}{132\,300}
-\frac{232}{14\,175}\ln 2
\]
\vspace{0.3cm}
\[
\int_{0}^{\infty}\frac{\sinh^{6}z}{z\,\cosh^{8}z}\,dz
= \Phi_1(7) =
-\frac{6}{7}\frac{\zeta'(2)}{\pi^{2}}
+28\frac{\zeta'(4)}{\pi^{4}}
-210\frac{\zeta'(6)}{\pi^{6}}
+510\frac{\zeta'(8)}{\pi^{8}}
+\frac{7\,943}{132\,300}
-\frac{428}{14\,175}\ln 2
\]

\vspace{1cm}

\[
\int_{0}^{\infty}\frac{\sinh^{2}z}{z\,\cosh^{3}z}\,dz
= \Phi_2(3) =
-2\frac{\beta'(1)}{\pi}
+16\frac{\beta'(3)}{\pi^{3}}
+\frac{3}{4}
\]
\vspace{0.3cm}
\[
\int_{0}^{\infty}\frac{\sinh^{2}z}{z\,\cosh^{5}z}\,dz
=
-\frac12\frac{\beta'(1)}{\pi}
-\frac{8}{3}\frac{\beta'(3)}{\pi^{3}}
+64\frac{\beta'(5)}{\pi^{5}}
+\frac{89}{288}
\]
\vspace{0.3cm}
\[
\int_{0}^{\infty}\frac{\sinh^{4}z}{z\,\cosh^{5}z}\,dz
=\Phi_2(5) =
-\frac{3}{2}\frac{\beta'(1)}{\pi}
+\frac{56}{3}\frac{\beta'(3)}{\pi^{3}}
-64\frac{\beta'(5)}{\pi^{5}}
+\frac{127}{288}
\]
\vspace{0.3cm}
\[
\int_{0}^{\infty}\frac{\sinh^{2}z}{z\,\cosh^{7}z}\,dz
=
-\frac{1}{4}\frac{\beta'(1)}{\pi}
-\frac{82}{45}\frac{\beta'(3)}{\pi^{3}}
+\frac{32}{3}\frac{\beta'(5)}{\pi^{5}}
+256\frac{\beta'(7)}{\pi^{7}}
+\frac{4\,201}{21\,600}
\]
\vspace{0.3cm}
\[
\int_{0}^{\infty}\frac{\sinh^{4}z}{z\,\cosh^{7}z}\,dz
=
-\frac{1}{4}\frac{\beta'(1)}{\pi}
-\frac{38}{45}\frac{\beta'(3)}{\pi^{3}}
+\frac{160}{3}\frac{\beta'(5)}{\pi^{5}}
-256\frac{\beta'(7)}{\pi^{7}}
+\frac{1\,237}{10\,800}
\]

\vspace{0.3cm}
\[
\int_{0}^{\infty}\frac{\sinh^{6}z}{z\,\cosh^{7}z}\,dz
=\Phi_2(7) =
-\frac{5}{4}\frac{\beta'(1)}{\pi}
+\frac{878}{45}\frac{\beta'(3)}{\pi^{3}}
-\frac{352}{3}\frac{\beta'(5)}{\pi^{5}}
+256\frac{\beta'(7)}{\pi^{7}}
+\frac{7\,051}{21\,600}
\]
\vspace{0.3cm}
\[
\int_{0}^{\infty}\frac{\sinh^{2}z}{z\,\cosh^{9}z}\,dz
=
-\frac{5}{32}\frac{\beta'(1)}{\pi}
-\frac{397}{315}\frac{\beta'(3)}{\pi^{3}}
+\frac{8}{15}\frac{\beta'(5)}{\pi^{5}}
+128\frac{\beta'(7)}{\pi^{7}}
+1\,024\frac{\beta'(9)}{\pi^{9}}
+\frac{4\,798\,639}{33\,868\,800}
\]

\vspace{0.5cm}

\textbf{Remark 1}

\vspace{0.3cm}

\medskip

\vspace{0.3cm}
No $\ln \pi$ appears in these last integrals. Strictly speaking, the following forms hold \begin{equation}
\begin{cases}
\displaystyle\int_0^{+\infty} \dfrac{\sinh^{2q}z}{z\, \cosh^{2n}z} \, dz = \displaystyle\sum_{p = 0}^{n-1}A_{p,q,n} \frac{\zeta'(2p+2)}{\pi^{2p+2}} + B_{q,n} +  C_{q,n}\ln 2 \\
\vspace{0.3cm}\\
\displaystyle\int_0^{+\infty} \dfrac{\sinh^{2q}z}{z\, \cosh^{2n+1}z} \, dz = \displaystyle\sum_{p = 0}^{n}D_{p,q,n} \frac{\beta'(2p+1)}{\pi^{2p+1}} + E_{q,n}
\label{eq:closed_forms_of_integrals}
\end{cases}
\end{equation}
where $A_{p,q,n}$, $B_{q,n}$, $C_{q,n}$, $D_{p,q,n}$, $E_{q,n}$ are some rational coefficients for all integers $q$ and $n$ satisfying the convergence condition. It is possible to prove very easily this statement for all higher $q$ and $n$ by identifying first the coefficients of $\ln \pi$ in the closed forms as
\[
-2q\,J_{q-1,n-1}+2n\,J_{q,n},
\qquad
-2q\,N_{q-1,n}+(2n+1)\,N_{q,n+1}.
\]

, second by using with the definitions of $J$ and $N$ given in \eqref{eq:formula_of_h} and then using \cref{lemma:euler_and_bernoulli} to show that these coefficients vanish.
\vspace{0.5cm}

\textbf{Remark 2}

\vspace{0.3cm}

\medskip

We also observe that the coefficient of \(\dfrac{\zeta'(2)}{\pi^{2}}\) in the integral \(\displaystyle\int_{0}^{\infty}\frac{\sinh^{2n-2} z}{z\,\cosh^{2n} z}\,dz\) is always equal to \(-\dfrac{6}{2n-1}\); the present proposition formalizes and establishes this result.

\vspace{0.5cm}

\begin{proposition}\label[proposition]{prop:coefficient_of_zeta_prime2}
\[\forall n \in \mathbb{N} : n \geq 2, \left[\dfrac{\zeta'(2)}{\pi^{2}}\right]\,\,\,\int_{0}^{\infty}\frac{\sinh^{2n-2}z}{z\,\cosh^{2n}z}\,dz = -\dfrac{6}{2n-1}\]
\end{proposition}

\begin{proof}

Let \(A_n := \left[\dfrac{\zeta'(2)}{\pi^{2}}\right]\,\,\,\displaystyle\int_{0}^{\infty}\dfrac{\sinh^{2n-2}z}{z\,\cosh^{2n}z}\,dz\). Or in other words, let \(A_n\) the coefficient of \(\dfrac{\zeta'(2)}{\pi^{2}}\) in the closed form of the integral \(\displaystyle\int_{0}^{\infty}\frac{\sinh^{2n-2} z}{z\,\cosh^{2n} z}\,dz\). Upon applying the substitutions \(q \mapsto n-1\) and \(n \mapsto 2n\) in~\eqref{eq:bridge_equality}, and combining~\eqref{eq:main_formulae} with~\eqref{eq:formula_of_h} followed by straightforward algebraic manipulations, we arrive at the conclusion that

\[
A_n
=
6(n-1)\sum_{m=1}^{n-1}\frac{c_{2n-2-2m,n-1}}{(2m-2)!}\,\Omega_{n-2,m-1}
-
6n\sum_{m=1}^{n}\frac{c_{2n-2m,n}}{(2m-2)!}\,\Omega_{n-1,m-1}.
\]
The second sum naturally suggests the definition
\[
T_n
:=6n\sum_{m=1}^{n}\frac{c_{2n-2m,n}}{(2m-2)!}\,\Omega_{n-1,m-1} = 6n\sum_{m=0}^{n-1}\frac{c_{2n-2m-2,n}}{(2m)!}\,\Omega_{n-1,m}.
\]
Then the first term equals $T_{n-1}$ and the second equals $T_n$, so that
\[
\,A_n=T_{n-1}-T_n.
\]

\Cref{lemma:from_sum_to_integral} allows us to conclude easily that
\[
T_n
=
\frac{6n}{2\pi i}
\oint_{\left|x\right|=\varepsilon}
\frac{\operatorname{artanh}^2 x}{x^{2n+1}}\,dx.
\]

By either the Cauchy coefficient formula or Cauchy's residue theorem,
\[
T_n
=6n\,[x^{2n}]\,\operatorname{artanh}^2 x.
\]
Using
\[
\operatorname{artanh}x=\sum_{m=0}^{\infty}\frac{x^{2m+1}}{2m+1}
=x\sum_{m=0}^{\infty}\frac{x^{2m}}{2m+1}
=:xR(x),
\]
we have $\operatorname{artanh}^2(x)=x^2R(x)^2$, and hence
\[
[x^{2n}]\,\operatorname{artanh}^2(x)=[x^{2n-2}]\,R(x)^2.
\]
Writing $R(x)=\displaystyle\sum_{m = 0}^{\infty}\dfrac{x^{2m}}{2m+1}$ leads to $R(\sqrt{x})=\displaystyle\sum_{m= 0}^{\infty}\dfrac{x^{m}}{2m+1}$ and applying the Cauchy product gives

\[R^2(\sqrt{x}) = \sum_{n = 0}^{\infty}\left[\sum_{m = 0}^{n} \frac{1}{(2m+1)(2n-2m+1)}\right]x^n\]

Or equivalently

\[R^2(x) = \sum_{n = 0}^{\infty}\left[\sum_{m = 0}^{n} \frac{1}{(2m+1)(2n-2m+1)}\right]x^{2n}\]

This gives finally
\[
[x^{2n-2}]\,R(x)^2
=
\sum_{m=0}^{n-1}\frac{1}{(2m+1)(2n-2m-1)}.
\]
Using
\[
\frac{1}{ab}=\frac{1}{a+b}\left(\frac{1}{a}+\frac{1}{b}\right),
\qquad a=2m+1,\quad b=2n-2m-1,\quad a+b=2n,
\]
one obtains
\[
[x^{2n}]\,\operatorname{artanh}^2(x)
= \begin{cases}
0&\quad \text{if} \quad n = 0, \\
\vspace{0.2cm}\\
\dfrac{1}{n}\displaystyle\sum_{k=1}^{n}\frac{1}{2k-1} &\quad \text{if} \quad n \geq 1.
\end{cases}\]
Therefore,
\[
T_n
=
6n\cdot \frac{1}{n}\sum_{k=1}^{n}\frac{1}{2k-1}
=
6\sum_{k=1}^{n}\frac{1}{2k-1} \quad \text{since} \quad n \geq 2.
\]

Recalling $A_n=T_{n-1}-T_n$, we conclude
\[
A_n
=
6\sum_{k=1}^{n-1}\frac{1}{2k-1}
-
6\sum_{k=1}^{n}\frac{1}{2k-1}
=
-\frac{6}{2n-1}\]
This concludes the proof.
\end{proof}

\begin{proposition}\label[proposition]{prop:another_non_trivial_identity} \[\forall n \in \mathbb{N}_0, \sum_{m=0}^{\,n}\frac{d_{2m,n+1}}{(2n-2m)!}
\sum_{k=0}^{\,n}\binom{4n+2}{\,2n-2k\,}\,(2k+1)^{\,2n-2m} = \frac{4^{n}}{2n+1}\]
\end{proposition} 

\begin{proof}
Show first of all that \[\sum_{k=0}^{\,n} \binom{4n+2}{\,2n-2k\,}\,
   \cosh\!\Big((2k+1)x\Big) = 16^{n} \left(\left(\cosh \dfrac{x}{2}\right)^{4n+2} + \left(\sinh \dfrac{x}{2}\right)^{4n+2}\right)\] and follow the same line of reasoning outlined in \cref{prop:vanishing} and \cref{prop:coefficient_of_zeta_prime2}
	\end{proof}
	
\section{\hspace{0.3cm}Vanishing functions with combinatorial coefficients}\label{sec:vanishing_functions}

\vspace{0.3cm}

\textit{This section is devoted to some functions $f_n$ of natural arguments which satisfy $f_n(j) = 0$ if $j \in \llbracket 0, n-1 \rrbracket$ or  $f_n(j) = 0$ if $j \in \llbracket 1, n-1 \rrbracket$. The designed functions have been derived in \textnormal{\cite{talla_waffo_integral_2025}} and involve some very well-known coefficients appearing in combinatorics.} If $n$ means a non zero natural integer, then there hold:

\[
\sum_{k=0}^{n} (-1)^k \binom{2n+1}{k} (2n+1 - 2k)^{2j+1} = 
\begin{cases}
0, & \text{if } j \in \llbracket 0, n - 1 \rrbracket\\
4^n (2n+1)!, & \text{if } j = n
\end{cases}
\]

\vspace{0.5cm}

\[\sum_{k=0}^{n-1} (-1)^k \binom{2n}{k} (2n-2k)^{2j} = 
\begin{cases}
(-1)^{n+1}\displaystyle\binom{2n}{n} , & \text{if } j = 0\\
0, & \text{if } j \in \llbracket 1, n - 1 \rrbracket \\
\dfrac{4^n (2n)!}{2}, & \text{if } j = n
\end{cases}
\]

\vspace{0.5cm}

\[ \sum_{m=0}^n \, c_{2m,n} \, \frac{1}{(2n-2m)!} 
 \,(2k+1)^{2n-2m} = \begin{cases}
0, & \text{if } k \in \llbracket 0, n - 1 \rrbracket \\
4^n , & \text{if } k = n
\end{cases}
\]

\vspace{0.5cm}

\[
\sum_{r=0}^{\,n-p}
\frac{c_{2n-2p-2r,n}}{(2r)!}
\sum_{k=0}^{\,n-1}
\genfrac{\langle}{\rangle}{0pt}{}{2n}{n-1-k}
\,(2k+1)^{2r} = \begin{cases}
0, & \text{if } p \in \llbracket 0, n - 1 \rrbracket \\
\dfrac{(2n)!}{2} , & \text{if } p = n
\end{cases}
\]

\vspace{0.5cm}

\[ \sum_{m=0}^{\,n-p}
      \frac{d_{2n-2m-2p,n}}{(2m)!}\,
      \sum_{k=0}^{n-1}
           \displaystyle\EulerB{2n-1}{k}\,
           (2n-1-2k)^{\,2m}\;=\;
 \begin{cases}
\dfrac{2^{2n-2}(2^{2n-1}-1)B_{2n}}{n}, & p = 0, \\
   0, & p \in \llbracket 1, n-1 \rrbracket,\\[6pt]
   2^{\,2n-2}\,(2n-1)!, & p=n.
 \end{cases}
\]

\vspace{0.5cm}

\[
\sum_{m=0}^{n}\frac{c_{2n-2m,n}}{(2m)!}
\sum_{k=0}^{q}\binom{2q+1}{k}\,(2q+1-2k)^{2m}=\begin{cases}
0, & \text{if } q \in \llbracket 0, n - 1 \rrbracket \\
4^n , & \text{if } q = n
\end{cases}
\]

\vspace{0.5cm}

\section{\hspace{0.3cm}Some Malmsten's integrals related to the polylogarithm}\label{sec:malmsten_integrals}

\[
\int_0^1 \frac{\Re\!\bigl(\operatorname{Li}_{-2n-1}(ix)\bigr)\,\ln\!\ln \frac{1}{x}}{x}\,dx
=
(-1)^n\,\frac{(2^{2n+1}-1)(2n)!}{2\pi^{2n}}\,\zeta(2n+1),
\qquad n\ge 1.
\]

\vspace{0.5cm}

\[
\int_0^1 \frac{\Re\!\bigl(\operatorname{Li}_{-1}(ix)\bigr)\,\ln\!\ln \frac{1}{x}}{x}\,dx
=
\frac12\!\left(\gamma+\ln\frac{4}{\pi}\right).
\]

\vspace{0.5cm}

\[
\int_0^1 \frac{\Im\!\bigl(\operatorname{Li}_{-2n-1}(ix)\bigr)\,\ln\!\ln \frac{1}{x}}{x}\,dx
=
\beta'(-2n)-\gamma\,\beta(-2n).
\]

\vspace{0.3cm}

\[
\int_0^1 \frac{\operatorname{Li}_{-2n}(-x^2)\,\ln\!\ln \frac{1}{x}}{x}\,dx
=
\frac12\Bigl((\gamma+\ln2)\,\eta(1-2n)-\eta'(1-2n)\Bigr),
\qquad n=0,1,2,\dots
\]

\vspace{0.3cm}

\[
\int_0^1 \frac{\Re\!\bigl(\operatorname{Li}_{-2n}(ix)\bigr)\,\ln\!\ln \frac{1}{x}}{x}\,dx
=
2^{\,2n-1}\Bigl((\gamma+\ln 2)\,\eta(1-2n)-\eta'(1-2n)\Bigr),
\qquad n=0,1,2,\dots
\]

\vspace{0.3cm}

\[
\int_0^1 \frac{\Re\!\bigl(\operatorname{Li}_{0}(ix)\bigr)\,\ln\!\ln \frac{1}{x}}{x}\,dx
=\frac{3}{4}\,\ln^2 2.
\]

\vspace{0.5cm}

\section{\hspace{0.3cm}{\bfseries\boldmath
The Mellin's transforms of $\dfrac{1}{\operatorname{arctanh} x}$ and $\dfrac{1}{\sqrt{1-x^2}\,\operatorname{arctanh} x}$
}
}\label{sec:transforms}

\vspace{0.5cm}

We remind \[\Phi_1(s) := \int_{0}^{1} \frac{x^{s-1}}{\operatorname{arctanh} x}\,dx, \qquad \Phi_2(s) := \int_{0}^{1} \frac{x^{s-1}}{\sqrt{1-x^2}\,\operatorname{arctanh} x}dx\]

\subsection{\hspace{0.3cm}Bounds and asymptotics}

\vspace{0.3cm}

We have \[\operatorname{arctanh} x = x + \dfrac{x^3}{3} + \dfrac{x^5}{5} + \dfrac{x^7}{7} + \cdots > x\] and

\[\operatorname{arctanh} x = x + \dfrac{x^3}{3} + \dfrac{x^5}{5} + \dfrac{x^7}{7} + \cdots < x + x^3 + x^5 + x^7 + \cdots = \dfrac{x}{1-x^2}\]

The starting point is consequently \[x < \operatorname{arctanh}x < \frac{x}{1-x^2}\]

Taking the inverse and multiplying yields

\[\begin{cases}
x^{s-2}(1-x^2) < \dfrac{x^{s-1}}{\operatorname{arctanh}x }< x^{s-2} \\
\vspace{0.3cm} \\
x^{s-2}\sqrt{1-x^2} < \dfrac{x^{s-1}}{\sqrt{1-x^2}\operatorname{arctanh}x } < \dfrac{x^{s-2}}{\sqrt{1-x^2}}
\end{cases}
\]

By integrating on (0,1) we arrive at the conclusion that 

\[\begin{cases}
\dfrac{2}{(s^2-1)}
\;<\;
\Phi_1(s)
\;<\;
\dfrac{1}{s-1} \\
\vspace{0.3cm} \\
\dfrac{\sqrt{\pi}}{2s}\,
\dfrac{\Gamma\!\left(\dfrac{s-1}{2}\right)}{\Gamma\!\left(\dfrac{s}{2}\right)}
\;<\;
\Phi_2(s)
\;<\;
\dfrac{\sqrt{\pi}}{2}\,
\dfrac{\Gamma\!\left(\dfrac{s-1}{2}\right)}{\Gamma\!\left(\dfrac{s}{2}\right)}
\end{cases}
\]

where \(\Gamma\) denotes the Gamma function. These identities readily
describe the qualitative behavior of the corresponding functions for
real values of \(s\). Moreover, they naturally motivate the
consideration of the following asymptotic expansions.

\[
\Phi_1(s)
=
\frac{1}{s-1}
+
C_1
+
O(s-1),
\qquad s\to 1^+,
\]
\[
\Phi_2(s)
=
\frac{1}{s-1}
+
C_2
+
O(s-1),
\qquad s\to 1^+,
\]
 where \begin{align*}
C_1
&=
\int_0^1
\left(
\frac{1}{\operatorname{arctanh}x}
-
\frac{1}{x}
\right)\,dx \\[0.3cm]
&= -6\dfrac{\zeta'(2)}{\pi^2} -1 + \ln \pi - \dfrac{4}{3} \ln 2\\[0.3cm]
&\approx -0.2095053618026607653
\end{align*}

\begin{align*}
C_2
&=
\int_0^1
\left(
\frac{1}{\sqrt{1-x^2}\,\operatorname{arctanh}x}
-
\frac{1}{x}
\right)\,dx\\[0.3cm]
&= -4\dfrac{\beta'(1)}{\pi} + \ln \dfrac{\pi}{2}\\[0.3cm]
& \approx 0.2059731205121406923
\end{align*}

\vspace{0.3cm}

\subsection{\hspace{0.3cm}Sums of beta as series of zeta and sums of zeta as series of beta}

\vspace{0.3cm}

\[
\Phi_1(s)
:= \int_{0}^{1} \frac{x^{s-1}}{\operatorname{arctanh}x}\,dx = \int_{0}^{1} \frac{x^{s-1}\sqrt{1-x^2}}{\sqrt{1-x^2}\operatorname{arctanh}x}\,dx,
\qquad
\Phi_2(s)
:= \int_{0}^{1} \frac{x^{s-1}}{\sqrt{1-x^2}\,\operatorname{arctanh}x}\,dx,
\]

\vspace{0.3cm}

For $|x|<1$ we use the binomial expansions
\[
\begin{cases}
\displaystyle
\frac{1}{\sqrt{1-x^2}}
=
\sum_{n=0}^{\infty}\binom{2n}{n}\frac{x^{2n}}{4^n},
\\[0.7cm]
\displaystyle
\sqrt{1-x^2}
=
1-\sum_{n=1}^{\infty}\frac{\binom{2n}{n}}{4^n(2n-1)}\,x^{2n}.
\end{cases}
\]

\vspace{0.4cm}

Insert the series into the definitions of $\Phi_1$ and $\Phi_2$:
\[
\begin{cases}
\displaystyle
\Phi_2(s)
=
\int_0^1
\frac{x^{s-1}}{\operatorname{arctanh}x}
\left(
\sum_{n=0}^{\infty}\binom{2n}{n}\frac{x^{2n}}{4^n}
\right)\,dx,
\\[0.8cm]
\displaystyle
\Phi_1(s)
=
\int_0^1
\frac{x^{s-1}}{\sqrt{1-x^2}\operatorname{arctanh}x}
\left(
1-\sum_{n=1}^{\infty}\frac{\binom{2n}{n}}{4^n(2n-1)}\,x^{2n}
\right)\,dx.
\end{cases}
\]

\vspace{0.4cm}

For $\Re(s)>1$ the integrands are integrable near $x=0$ and the series converge
uniformly on $[0,1-\varepsilon]$, hence summation and integration may be interchanged:
\[
\begin{cases}
\displaystyle
\Phi_2(s)
=
\sum_{n=0}^{\infty}\binom{2n}{n}\frac{1}{4^n}
\int_0^1 \frac{x^{s+2n-1}}{\operatorname{arctanh}x}\,dx,
\\[0.8cm]
\displaystyle
\Phi_1(s)
=
\int_0^1
\frac{x^{s-1}}{\sqrt{1-x^2}\operatorname{arctanh}x}\,dx
-
\sum_{n=1}^{\infty}\frac{\binom{2n}{n}}{4^n(2n-1)}
\int_0^1
\frac{x^{s+2n-1}}{\sqrt{1-x^2}\operatorname{arctanh}x}\,dx.
\end{cases}
\]

\vspace{0.4cm}

Using the definitions of $\Phi_1$ and $\Phi_2$ again, we obtain the coupled system
\[
\boxed{
\begin{cases}
\displaystyle
\Phi_2(s)
=
\sum_{n=0}^{\infty}
\binom{2n}{n}\frac{1}{4^n}\,\Phi_1(s+2n),
\\[0.6cm]
\displaystyle
\Phi_1(s)
=
-
\sum_{n=0}^{\infty}
\frac{\binom{2n}{n}}{4^n(2n-1)}\,\Phi_2(s+2n),
\end{cases}
\qquad \Re(s)>1.
}
\]

\vspace{0.3cm}

For few first even $s$ see \cite[Sec 4, §6, 82-84]{talla_waffo_integral_2025} and \cite[Sec 5, §5, 115-117]{talla_waffo_integral_2025}, one has

\begin{align*}
\frac{7\zeta(3)}{\pi^2}
&=
\frac{5\,386\,925}{3\,407\,872}\frac{\beta(2)}{\pi}
+
\frac{1\,492\,525\,919}{94\,617\,600}\frac{\beta(4)}{\pi^{3}}
-
\frac{3\,669\,179}{92\,160}\frac{\beta(6)}{\pi^{5}}
+
\frac{357\,259}{2\,880}\frac{\beta(8)}{\pi^{7}}\\
&\quad
-
\frac{11\,967}{40}\frac{\beta(10)}{\pi^{9}}
+
462\,\frac{\beta(12)}{\pi^{11}}
-
336\,\frac{\beta(14)}{\pi^{13}}
\\
&\quad
-
\sum_{n=7}^{\infty}
\frac{\binom{2n}{n}}{4^n(2n-1)}\,\Phi_2(2n+2)
\end{align*}

\vspace{0.5cm}

\begin{align*}
\frac{14}{3}\frac{\zeta(3)}{\pi^{2}}
-
31\,\frac{\zeta(5)}{\pi^{4}}
&=
\frac{11\,197\,885}{9\,371\,648}\frac{\beta(2)}{\pi}
+
\frac{56\,749\,463}{14\,192\,640}\frac{\beta(4)}{\pi^{3}}
-
\frac{120\,684\,359}{1\,451\,520}\frac{\beta(6)}{\pi^{5}}
+
\frac{80\,843}{432}\frac{\beta(8)}{\pi^{7}}\\
&\quad
-
\frac{4\,189}{10}\frac{\beta(10)}{\pi^{9}}
+
\frac{1\,880}{3}\frac{\beta(12)}{\pi^{11}}
-
448\,\frac{\beta(14)}{\pi^{13}}
\\
&\quad
-
\sum_{n=6}^{\infty}
\frac{\binom{2n}{n}}{4^n(2n-1)}\,\Phi_2(2n+4)
\end{align*}

\vspace{0.5cm}

\begin{align*}
\frac{161}{45}\frac{\zeta(3)}{\pi^{2}}
-
\frac{124}{3}\frac{\zeta(5)}{\pi^{4}}
+
127\,\frac{\zeta(7)}{\pi^{6}}
&=
\frac{44\,248\,103}{42\,172\,416}\frac{\beta(2)}{\pi}
-
\frac{200\,799\,329}{106\,444\,800}\frac{\beta(4)}{\pi^{3}}
-
\frac{9\,160\,721}{145\,152}\frac{\beta(6)}{\pi^{5}}\\
&\quad
+
\frac{415\,337}{1\,080}\frac{\beta(8)}{\pi^{7}}
-
659\,\frac{\beta(10)}{\pi^{9}}
+
\frac{2\,768}{3}\frac{\beta(12)}{\pi^{11}}
-
640\,\frac{\beta(14)}{\pi^{13}}
\\
&\quad
-
\sum_{n=5}^{\infty}
\frac{\binom{2n}{n}}{4^n(2n-1)}\,\Phi_2(2n+6)
\end{align*}

\vspace{0.5cm}

\begin{align*}
\frac{4\,\beta(2)}{\pi}
&=\frac{87\,350\,741}{6\,589\,440}\,\frac{\zeta(3)}{\pi^{2}}
-\frac{13\,911\,343}{172\,800}\,\frac{\zeta(5)}{\pi^{4}}
+\frac{10\,591\,927}{23\,040}\,\frac{\zeta(7)}{\pi^{6}}\\
&\quad
-\frac{1\,1093\,299}{5\,760}\,\frac{\zeta(9)}{\pi^{8}}
+\frac{5602639}{1\,024}\,\frac{\zeta(11)}{\pi^{10}}
-\frac{1\,204\,077}{128}\,\frac{\zeta(13)}{\pi^{12}}
+\frac{7\,569\,177}{1\,024}\,\frac{\zeta(15)}{\pi^{14}}\\
&\quad
\qquad
+\sum_{n=7}^{\infty}\binom{2n}{n}\frac{1}{4^n}\,\Phi_1(2n+2).
\end{align*}

\vspace{0.5cm}

\begin{align*}
\frac{10}{3}\,\frac{\beta(2)}{\pi}-16\,\frac{\beta(4)}{\pi^{3}}
&=\frac{9\,308\,719}{988\,416}\,\frac{\zeta(3)}{\pi^{2}}
-\frac{2\,158\,068\,007}{19\,958\,400}\,\frac{\zeta(5)}{\pi^{4}}
+\frac{9\,143\,873}{17\,280}\,\frac{\zeta(7)}{\pi^{6}}\\
&\quad
-\frac{1\,547\,819}{720}\,\frac{\zeta(9)}{\pi^{8}}
+\frac{4\,632\,361}{768}\,\frac{\zeta(11)}{\pi^{10}}
-\frac{1\,318\,751}{128}\,\frac{\zeta(13)}{\pi^{12}}
+\frac{2\,064\,321}{256}\,\frac{\zeta(15)}{\pi^{14}}\\
&\quad
+\sum_{n=6}^{\infty}\binom{2n}{n}\frac{1}{4^n}\,\Phi_1(2n+4).
\end{align*}

\vspace{0.5cm}

\begin{align*}
\frac{89}{30}\,\frac{\beta(2)}{\pi}-24\,\frac{\beta(4)}{\pi^{3}}+64\,\frac{\beta(6)}{\pi^{5}}
&=\frac{88\,983\,991}{12\,355\,200}\,\frac{\zeta(3)}{\pi^{2}}
-\frac{27\,337\,753}{249\,480}\,\frac{\zeta(5)}{\pi^{4}}
+\frac{13\,701\,649}{20\,160}\,\frac{\zeta(7)}{\pi^{6}}\\
&\quad
-\frac{1\,075\,655}{432}\,\frac{\zeta(9)}{\pi^{8}}
+\frac{2\,618\,113}{384}\,\frac{\zeta(11)}{\pi^{10}}
-\frac{368\,595}{32}\,\frac{\zeta(13)}{\pi^{12}}
+\frac{1\,146\,845}{128}\,\frac{\zeta(15)}{\pi^{14}}\\
&\quad
+\sum_{n=5}^{\infty}\binom{2n}{n}\frac{1}{4^n}\,\Phi_1(6+2n).
\end{align*}

\vspace{0.5cm}

\begin{align*}
\frac{381}{140}\,\frac{\beta(2)}{\pi}
-\frac{434}{15}\,\frac{\beta(4)}{\pi^{3}}
+\frac{416}{3}\,\frac{\beta(6)}{\pi^{5}}
-256\,\frac{\beta(8)}{\pi^{7}}
&=\frac{20\,241\,929}{3\,603\,600}\,\frac{\zeta(3)}{\pi^{2}}
-\frac{6\,040\,691}{59\,400}\,\frac{\zeta(5)}{\pi^{4}}
+\frac{5\,856\,097}{7\,560}\,\frac{\zeta(7)}{\pi^{6}}\\
&\quad
-\frac{432\,671}{135}\,\frac{\zeta(9)}{\pi^{8}}
+\frac{128\,961}{16}\,\frac{\zeta(11)}{\pi^{10}}
-\frac{106\,483}{8}\,\frac{\zeta(13)}{\pi^{12}}\\
&\quad
+\frac{163\,835}{16}\,\frac{\zeta(15)}{\pi^{14}}
+\sum_{n=4}^{\infty}\binom{2n}{n}\frac{1}{4^n}\,\Phi_1(8+2n).
\end{align*}

\vspace{0.5cm}

\subsection{\hspace{0.3cm}A new perspective on the arithmetic nature of the ratios $\dfrac{\zeta(2n+1)}{\pi^{2n+1}}$ and $\dfrac{\beta(2n)}{\pi^{2n}}$}

\vspace{0.3cm}

For $\Re(s)>1$ we consider
\[
\Phi_1(s)
:=\int_0^1
\frac{x^{s-1}\sqrt{1-x^2}}{\sqrt{1-x^2}\,\operatorname{arctanh}x}\,dx, \Phi_2(s)
:=\int_0^1
\frac{x^{s-1}}{\sqrt{1-x^2}\,\operatorname{arctanh}x}\,dx .
\]

Assume the power–series expansions
\[
\begin{cases}
\displaystyle
\frac{1}{\operatorname{arctanh}x}
=\sum_{n=0}^{\infty} p_n\,x^{2n-1},
\\[0.6cm]
\displaystyle
\frac{1}{\sqrt{1-x^2}\,\operatorname{arctanh}x}
=\sum_{n=0}^{\infty} q_n\,x^{2n-1},
\end{cases}
\qquad 0<|x|<1.
\] with coefficients \(p_n, q_n\) defined implicitly by these identities

\vspace{0.4cm}

Using the representations
\[
\begin{cases}
\displaystyle
\Phi_2(s)
=
\int_0^1
\frac{x^{s-1}}{\sqrt{1-x^2}\,\operatorname{arctanh}x}\,dx,
\\[0.6cm]
\displaystyle
\Phi_1(s)
=
\int_0^1
\frac{x^{s-1}\sqrt{1-x^2}}{\sqrt{1-x^2}\,\operatorname{arctanh}x}\,dx,
\end{cases}
\]
we substitute the corresponding series and interchange summation and integration
(for \(\Re(s)>1\)):

\[
\begin{cases}
\displaystyle
\Phi_2(s)
=
\sum_{n=0}^{\infty} p_n
\int_0^1
\frac{x^{s+2n-2}}{\sqrt{1-x^2}}\,dx,
\\[0.8cm]
\displaystyle
\Phi_1(s)
=
\sum_{n=0}^{\infty} q_n
\int_0^1
x^{s+2n-2}\sqrt{1-x^2}\,dx.
\end{cases}
\]

\vspace{0.4cm}

The integrals are evaluated in terms of Beta functions:
\[
\begin{cases}
\displaystyle
\int_0^1
\frac{x^{s+2n-2}}{\sqrt{1-x^2}}\,dx
=
\frac{\sqrt{\pi}}{2}
\frac{\Gamma\!\left(\frac{s+2n-1}{2}\right)}
{\Gamma\!\left(\frac{s+2n}{2}\right)},
\\[1.0cm]
\displaystyle
\int_0^1
x^{s+2n-2}\sqrt{1-x^2}\,dx
=
\frac{\sqrt{\pi}}{4}
\frac{\Gamma\!\left(\frac{s+2n-1}{2}\right)}
{\Gamma\!\left(\frac{s+2n+2}{2}\right)}.
\end{cases}
\]

\vspace{0.4cm}

Using Legendre’s duplication formula\[
\Gamma(z)\,\Gamma\!\left(z+\tfrac12\right)
=2^{\,1-2z}\sqrt{\pi}\,\Gamma(2z),
\] both expressions reduce to
\[
\boxed{
\begin{cases}
\displaystyle
\Phi_2(s)
=
\pi\,2^{\,1-s}
\sum_{n=0}^{\infty}
p_n\,4^{-n}\,
\frac{\Gamma(s+2n-1)}
{\Gamma\!\left(\frac{s+2n}{2}\right)^2},
\\[1.0cm]
\displaystyle
\Phi_1(s)
=
\pi\,2^{\,1-s}
\sum_{n=0}^{\infty}
q_n\,4^{-n}\,
\frac{\Gamma(s+2n-1)}
{(s+2n)\,\Gamma\!\left(\frac{s+2n}{2}\right)^2},
\end{cases}
\qquad \Re(s)>1.
}
\]

\vspace{0.5cm}

Specializing to even arguments \(s=2m\), \(m=1,2,3,\dots\), we finally obtain
\[
\boxed{
\begin{cases}
\displaystyle
\Phi_2(2m)
=
\pi\,2^{\,1-2m}
\sum_{n=0}^{\infty}
p_n\,4^{-n}\,
\binom{2(m+n-1)}{\,m+n-1\,},
\\[1.0cm]
\displaystyle
\Phi_1(2m)
=
\pi\,2^{\,1-2m}
\sum_{n=0}^{\infty}
q_n\,4^{-n}\,
\frac{1}{2(m+n)}\,
\binom{2(m+n-1)}{\,m+n-1\,},
\end{cases}
\qquad m=1,2,3,\dots
}
\]

We already established that the irrationality of $\dfrac{\beta(2n)}{\pi^{2n}}$ would follow from this conjecture \cite{talla_waffo_integral_2025} \cite{TallaWaffo2025arxiv2511.02843}:
\[
\pi \notin \text{Span}_{\mathbb{Q}}\left\{
\Phi_2(2),\ 
\Phi_2(4),\ 
\Phi_2(6),\ 
\Phi_2(8),\ 
\Phi_2(10),\ 
\Phi_2(12),\ 
\Phi_2(14),\ 
\Phi_2(16),\ 
\Phi_2(18),\ 
\ldots \right\}
\]

\vspace{0.3cm}

In virtue of \[
\Phi_2(2m)
=
\pi\,2^{\,1-2m}
\sum_{n=0}^{\infty}
p_n\,4^{-n}\,
\binom{2(m+n-1)}{\,m+n-1\,},
\qquad m=1,2,3,\dots
\], the irrationality of $\dfrac{\beta(2n)}{\pi^{2n}}$ would follow from this conjecture since $2^{1-2m}$ is always rational: 

\[
1 \notin \text{Span}_{\mathbb{Q}} \left\{ 
\sum_{n=0}^{\infty}
p_n\,4^{-n}\,
\binom{2n}{\,n\,},\ 
\sum_{n=0}^{\infty}
p_n\,4^{-n}\,
\binom{2n+2}{\,n+1\,},\ 
\sum_{n=0}^{\infty}
p_n\,4^{-n}\,
\binom{2n+4}{\,n+2\,},\ 
\sum_{n=0}^{\infty}
p_n\,4^{-n}\,
\binom{2n+6}{\,n+3\,},\  
\ldots \right\}
\] 

\vspace{0.3cm}

Similarly, we already established that the irrationality of 
$\dfrac{\zeta(2n+1)}{\pi^{2n+1}}$ would follow from the conjecture
\cite{talla_waffo_integral_2025,TallaWaffo2025arxiv2511.02843}
\[
\pi \notin \operatorname{Span}_{\mathbb{Q}}\!\left\{
\Phi_1(2),\,
\Phi_1(4),\,
\Phi_1(6),\,
\Phi_1(8),\,
\Phi_1(10),\,
\Phi_1(12),\,
\Phi_1(14),\,
\Phi_1(16),\,
\Phi_1(18),\,
\ldots
\right\}.
\]

Using the representation
\[
\Phi_1(2m)
=
\pi\,2^{\,1-2m}
\sum_{n=0}^{\infty}
q_n\,4^{-n}\,
\frac{1}{2(m+n)}
\binom{2(m+n-1)}{\,m+n-1\,},
\qquad m=1,2,3,\dots,
\]
this conjecture is equivalently reformulated as
\[
1 \notin \operatorname{Span}_{\mathbb{Q}}\!\left\{
\sum_{n=0}^{\infty} q_n\,4^{-n}\frac{1}{n+1}\binom{2n}{n},\,
\sum_{n=0}^{\infty} q_n\,4^{-n}\frac{1}{n+2}\binom{2n+2}{n+1},\,
\sum_{n=0}^{\infty} q_n\,4^{-n}\frac{1}{n+3}\binom{2n+4}{n+2},\,
\ldots
\right\}.
\]

\vspace{0.5cm}

\printbibliography

\end{document}